\pdfoutput=1

\documentclass{article}

\usepackage{amsmath}
\usepackage{amsthm}
\newcommand{\norm}[1]{\left\lVert#1\right\rVert}
\newcommand{\abs}[1]{\left\lvert#1\right\rvert}
\newcommand{\inn}[2]{\left\langle#1\mid#2\right\rangle}
\newcommand{\set}[2]{\left\lbrace #1 \middle| #2 \right\rbrace}

\DeclareMathAlphabet{\mathbbold}{U}{bbold}{m}{n}

\usepackage{amssymb,amsfonts}
\usepackage[all,arc]{xy}
\usepackage{enumerate}
\usepackage{mathrsfs}
\usepackage{hyperref}
\usepackage[margin=1.25in]{geometry}
\usepackage{pgf,tikz,pgfplots}
\usepackage{graphicx}
\usepackage{accents}
\usepackage{mathtools}

\usepackage[title]{appendix}
\usepackage{listings}
\usepackage{color}
\definecolor{mygreen}{RGB}{28,172,0} 
\definecolor{mylilas}{RGB}{170,55,241}

\tikzset{
  shaded/.style = {fill=red!10!blue!20!gray!30!white},
  unshaded/.style = {fill=white},
  shadedw/.style = {fill=white},
  shadedb/.style = {fill=blue!120!gray!30!white},
  shadedr/.style = {fill=red!120!gray!30!white},
  shadedg/.style = {fill=green!120!gray!30!white},
  shadedy/.style = {fill=yellow!120!gray!30!white},
  Tcirc/.style = {circle, draw, thick, fill=white, opaque},
  Tellip/.style = {ellipse, draw, thick, fill=white, opaque},
  Tbox/.style = {rounded corners,rectangle, draw, thick, fill=white, opaque},
  align.7/.style = {scale=.7, baseline},
  align1/.style = {scale=1.05, baseline},
  align1.5/.style = {scale=1.5,baseline},
  every picture/.style=semithick
}

\usetikzlibrary{
  snakes,
  knots,
  hobby,
  decorations.pathreplacing,
  shapes.geometric,
  calc
}

\tikzset{
  knot diagram/every strand/.append style={black, thick},
  stock/.style={consider self intersections=true, end tolerance=1pt, clip width=5pt, clip radius=5pt},
  stockthick/.style={consider self intersections=true, end tolerance=1pt, clip width=3pt, clip radius=20pt},
  shaded/.style = {fill=red!10!blue!20!gray!30!white},
  unshaded/.style = {fill=white},
}
\pgfplotsset{compat=1.17}

\newtheorem{thm}{Theorem}[section]

\newtheorem{prop}[thm]{Proposition}
\newtheorem{lem}[thm]{Lemma}
\newtheorem{conj}[thm]{Conjecture}

\newtheorem{defn}[thm]{Definition}

\newtheorem{exmp}[thm]{Example}

\newtheorem{rem}[thm]{Remark}

\title{The spectrum of spin model angle operators}
\author{Michael Montgomery}
\begin{document}
\maketitle
\begin{abstract}
	Complex Hadamard matrices are biunitaries for spin model commuting squares. The corresponding subfactor standard invariant can be identified with the $1$-eigenspace of the angle operator defined by Jones. We identify the angle operator as an element of the symmetric enveloping algebra and compute its trace. We then show the angle operator spectrum coincides with the principal graph spectrum up to a constant iff the subfactor is amenable. We use this to show Paley type $II$ Hadamard matrices and Petrescu's $7 \times 7$ family of complex Hadamard matrices yield infinite depth subfactors.
\end{abstract}

\vspace{7mm}

Jones introduced the basic construction of a finite index subfactor, $N\subset M$, in \cite{Jon83} . Iterating his construction yields the Jones tower of $II_1$ factors $$M_{-1} = N \subset M_0=M \subset M_1 \subset M_2 \subset \cdots .$$ Taking relative commutants of these factors we may build the standard invariant of the subfactor which consists of finite dimensional $C^*$-algebras, $\set{M_i' \cap M_n}{i=-1,0, \text{ } n=-1,0,1,...}$, inclusions $M_0' \cap M_n \subset M_{-1}' \cap M_n$, $M_i' \cap M_n \subset M_i' \cap M_{n+1}$, and Jones projections $\set{e_n}{ n \geq 1}$. Classifying standard invariants and constructing exotic examples has been a multi-decade project that has contributed to low-dimensional topology and many areas of mathematical physics. A summary of this can be found in \cite{JMS14}, and for an introduction to subfactors see \cite{JS97}. Jones axiomatized the standard invariant of extremal subfactors as subfactor planar algebras in \cite{Jon99}. Weaker invariants of $N \subset M$ can be constructed from the Bratteli diagrams of the standard invariant, called the principal graphs of $N \subset M$.

In this paper we focus on subfactors generated from complex Hadamard matrices called spin model subfactors. The principal graphs of subfactors for twisted tensor products of Fourier matrices have been identified by Burstein as Bisch-Haagerup subfactors, but very little is known about the standard invariants of spin model subfactors outside of this family (see \cite{Bur15} and \cite{BH96}). In \cite{Jon99}, Jones defined an angle operator, $\Theta_u$, in the sense of \cite{SW94}, whose $1$-eigenspace is the standard invariant of the subfactor. We identify this angle operator as an element of $C^*(M,e_N,JMJ)$, the $C^*$-algebra generated by $M$, $e_N$, and $JMJ$ on $L^2(M)$. Popa showed in \cite{Pop99} that $C^*(M,e_N,JMJ)$ admits a tracial state, $\tau$, which is faithful iff the subfactor is amenable. We then compute $\tau(\Theta_u^n)$ in terms of the standard invariant and prove a correspondence between the principal graph spectrum and angle operator spectrum. Since the angle operator has finite dimensional representations we can compute elements of its spectrum.

Finally, we find non-algebraic integers in the spectra of angle operators for Petrescu's continuous family of $7\times 7$ complex Hadamard matrices \cite{Pet97} and Paley type $II$ Hadamard matrices \cite{Pal33}. Since the spectrum of finite graphs only contain algebraic integers, these subfactors are infinite depth. Jones asked in \cite{Jon99} if any complex Hadamard matrix yields an $A_\infty$ principal graph. This question remains unresolved, but we conjecture that Paley type $II$ Hadamard matrices yield subfactors with $A_\infty$ principal graphs.

\section{Preliminaries}

We will work in a planar algebra called $P^{Spin}$ defined in \cite{Jon99} and \cite{Jon19}. We first define the spin planar algebra.

\begin{defn}\normalfont (\cite{Jon99},\cite{Jon19})
A \textit{shaded planar tangle} $T$ consists of the following data:
		\begin{enumerate}[(i)]
			\item A smooth disc $D^T \subset \mathbb R^2$ called the \textit{output disc}.
			\item A finite collection of disjoint smooth discs $\mathfrak D_T$ that lie inside $Int(D^T)$ called the \textit{input discs}.
			\item A finite collection of disjoint smooth curves $\mathfrak S_T$ that lie inside $D^T - \bigcup_{D \in \mathfrak D_T} Int(D)$ such that its boundary points belong to the input discs or the boundary of the output disc and all curves meet discs transversely if at all. Elements of $\mathfrak S_T$ are called \textit{strings} of $T$.
		\item An assignment of shaded or unshaded to the connected regions of $D^T -\left(\displaystyle\bigcup_{\mathfrak D_T} D\right)\bigcup \left(\displaystyle\bigcup_{\mathfrak S_T} s \right)$ such that every string belongs to the boundary of a shaded region and an unshaded region.
		\item The boundary of each disc is broken into a finite number of components called the \textit{boundary points} of $D$, the points in $(\bigcup_{s \in \mathfrak S_T} s)\cap \partial D$, and the \textit{intervals} of $D$, the connected components of $\partial D - \bigcup_{s \in \mathfrak S_T} s$. Each disc has a single marked interval that we will denote with a $\$$. We will assign to each disc $D$ \textit{boundary data} $\partial D =(n_D,\pm)$ where $n_D := \frac{1}{2}\#(\text{boundary points of D})$ and $\pm$ indicates whether the marked interval is shaded.
		\end{enumerate}
	\end{defn}

	\begin{defn}\normalfont (\cite{Jon99},\cite{Jon19})
		A $\textit{unital shaded planar algebra}$ $P$ is a family of vector spaces $P_{n,\pm}, n \in \mathbb N\cup \{0\}$ with multilinear maps
		$$Z_T \colon \bigtimes_{D \in \mathfrak D_T} P_{\partial D} \to P_{\partial D^T}$$
		for every planar tangle $T$ such that:
		\begin{enumerate}[(i)]
			\item If $\theta$ is an orientation preserving diffeomorphism of $\mathbb R^2$, then
			$Z_{\theta(T)}(f)=Z_T(f \circ \theta)$.
			\item $Z_{T \circ S} = Z_T \circ Z_S$ where 
			$Z_T \circ Z_S (f) = Z_T(\tilde{f})$ and $\tilde{f}=
			\left\lbrace \begin{array}{cc}
				f(D) & \text{if } D \neq D^S\\
				Z_S(f\mid_{\mathfrak D_S}) & \text{if } D=D^S\\
			\end{array}\right.$.
		\end{enumerate}
		Elements of $\displaystyle\bigtimes_{D \in \mathfrak D_T} P_{\partial D}$ are called \textit{labellings} of the tangle $T$.
	\end{defn}

\begin{defn}\normalfont
In \cite{Jon19} Jones defines the following planar algebra called the \textit{spin planar algebra}. Fix $Q \in \mathbb N$. Let $P_{0,+}^{Spin} = \mathbb C$, $P_{0,-}^{Spin} =\mathbb C^Q$ and $P_{n,\pm}^{Spin} = (\mathbb C^Q)^{\otimes n}$ where elements in $P_{n, \pm}^{Spin}$ correspond to a disc with $2n$ boundary points and $n$ shaded intervals. Fix a basis $B=\{\hat{1},...,\hat{Q}\}$ on $\mathbb C^Q$. Vectors $v \in (\mathbb C^Q)^{\otimes n}$ are a (unique) linear combination of simple tensors $B^{\otimes n} = \set{\otimes_{i=1}^n \hat{s}_i}{\hat{s}_i \in B}$, $v=\displaystyle\sum_{\hat{b} \in B^{\otimes n}} v_b \hat{b}$, $v_b \in \mathbb C$. Then the following rules will equip these vector spaces with a unital shaded $*$-planar algebra structure.
\begin{enumerate}[(i)]
\item A state on a shaded planar tangle, $T$, is a map $\sigma \colon \{\text{connected shaded regions of }T\} \to \{\hat{1},...,\hat{Q}\}$.
\item For each disc $D \in \mathfrak D_T$, a state $\sigma$ induces the following labelling of shaded intervals of $D$. Count the shaded intervals of $D$ in a counter clockwise direction starting after the marked interval. Set $\sigma_D =\otimes_{i=1}^{n_D}\hat{s}_i$ where the $i^{th}$ shaded interval belongs to the boundary of a region labelled by $\hat{s}_i \in B$. If $D$ has no shaded intervals then $\sigma_D = 1$.
\item Given a shaded planar tangle $T$ and a region $r$ in $T$, define $Rot(r)$ as follows: Remove input disks with zero boundary points. Then give $r$ a counter-clockwise orientation inducing an orientation on the boundary of $r$ which is a union of piecewise smooth curves. Define $Rot(r)$ as the rotation number of the oriented boundary of $r$.
\item Given a state $\sigma$ on a shaded planar tangle $T$, define 
$$Rot(\sigma)= \prod_{\genfrac{}{}{0pt}{3}{\text{shaded regions}}{\text{$r$ of $T$}}} \left(\dfrac{1}{\sqrt{Q}}\right)^{Rot(r)}.$$
\item Define $(\otimes_{i=1}^n \hat{s}_i)^*=\otimes_{i=1}^n \hat{s}_{n-i+1}$ for $\otimes_{i=1}^n \hat{s}_i \in P^{Spin}_{n,+}$ and $(\otimes_{i=1}^n \hat{s}_i)^*=\otimes_{i=1}^{n-1} \hat{s}_{n-i} \otimes \hat{s}_n$ for\\
$\otimes_{i=1}^n \hat{s}_i \in P^{Spin}_{n,-}$ then extend $*$ to $P^{Spin}_{n,\pm}$ by conjugate linearity.

\item Finally, we define the action of a tangle $T$ with a labelling $f \in \displaystyle\bigtimes_{D \in \mathfrak D_T} P^{Spin}_{\partial D}$ by 
$$Z_T(f) = \sum_\sigma \sqrt{Q}^{n_T} Rot(\sigma) \prod_{D \in \mathfrak D_T} f(D)_{\sigma_D}  \sigma_{D^T},$$
where $n_T$ is the number of shaded intervals in the output disk, $f(D)=\displaystyle\sum_{\hat{b} \in B^{\otimes n}} f(D)_b \hat{b}$ for $n=n_D$, and an empty product is interpreted as $1$.
\end{enumerate}
\end{defn}

For example, fix $Q=5$ and let $x=\hat{1} \otimes \hat{2}$, $y=\hat{2} \otimes \hat{3}$. Then we can evaluate the tangle
$$\begin{tikzpicture}[align1]
\clip[rounded corners] (0,-.5) rectangle (2,.6);
\draw[shaded] (0,-.1) rectangle (2,.1);
\draw[shaded] (1,-.5) circle (.2cm);
\node at (.6,0)[Tbox, minimum width=.5cm, minimum height=.5cm]{\scriptsize $x$};
\node at (1.4,0)[Tbox, minimum width=.5cm, minimum height=.5cm]{\scriptsize $y$};
\node at (1,0){\scriptsize $b$};
\node at (.2,0){\scriptsize $a$};
\node at (1.8,0){\scriptsize $c$};
\node at (1,-.4){\scriptsize $d$};
\node at (.1,.45) {\scriptsize $\$$};
\node at (.6,.35) {\scriptsize $\$$};
\node at (1.4,.35) {\scriptsize $\$$};
\draw[very thick, rounded corners] (0,-.5) rectangle (2,.6);
\end{tikzpicture}=\sum_{a,b,c,d} \sqrt{5}^3 \left(\frac{1}{\sqrt{5}}\right)^4 x_{\hat{a}\otimes\hat{b}} y_{\hat{b}\otimes \hat{c}} \hat{a}\otimes\hat{d}\otimes \hat{c}=\frac{1}{\sqrt{5}}\sum_{d=1}^5 \hat{1}\otimes\hat{d}\otimes\hat{3}.$$

Jones defined another planar algebra, $P^{spin}$ also called the spin planar algebra, where the action by tangles is nearly identical. It can be obtained by removing $\sqrt{Q}^{n_T} Rot(\sigma)$ from the definition above. We work with $P^{Spin}$ in this paper since both of its loop parameters $\delta_+$ and $\delta_-$ are equal and the type $II$ Reidemeister moves we will perform later have a cleaner presentation. 

\begin{prop}(\cite{Jon99},\cite{Jon19})
$P^{Spin}$ is a shaded $C^*$-planar algebra i.e. each $P^{Spin}_{n,\pm}$ becomes a $C^*$-algebra with the multiplication tangle
$$xy=\begin{tikzpicture}[align1]
\clip[rounded corners] (0,-.5) rectangle (2,.6);
\draw[line width=1mm] (0,0)--(2,0);
\node at (.6,0)[Tbox, minimum width=.5cm, minimum height=.5cm]{\scriptsize $x$};
\node at (1.4,0)[Tbox, minimum width=.5cm, minimum height=.5cm]{\scriptsize $y$};
\node at (1,0)[Tcirc, inner sep=0]{\scriptsize $n$};
\node at (.2,0)[Tcirc, inner sep=0]{\scriptsize $n$};
\node at (1.8,0)[Tcirc, inner sep=0]{\scriptsize $n$};
\node at (.1,.45) {\scriptsize $\$$};
\node at (.6,.35) {\scriptsize $\$$};
\node at (1.4,.35) {\scriptsize $\$$};
\draw[very thick, rounded corners] (0,-.5) rectangle (2,.6);
\end{tikzpicture}$$
and the $*$-operation. Furthermore, $P^{Spin}$ has loop parameters
			$$\delta_+ =
			\begin{tikzpicture}[align1,scale=.5]
				\clip (0,0) circle (.75cm);
				\begin{scope}[shift=(-90:4mm)]
					\draw[shaded] (0,0) to [out=180, in=160] (0,6mm) to [out=-20 in=160] (1.8mm,9mm) to [out=-20, in=0] (0,0);
				\end{scope}
				\node at (-5.5mm,0) {\scriptsize $\$$};
				\draw[very thick] (0,0) circle (.75cm);
			\end{tikzpicture}=\sqrt{Q}\cdot\text{id}
		\quad \text{ and }\quad \delta_-=
			\begin{tikzpicture}[align1, scale=.5]
				\clip (0,0) circle (.75cm);
				\draw[shaded, very thick] (0,0) circle (.75cm);
				\begin{scope}[shift=(-90:4mm)]
					\draw[unshaded] (0,0) to [out=180, in=160] (0,6mm) to [out=-20 in=160] (1.8mm,9mm) to [out=-20, in=0] (0,0);
				\end{scope}
				\node at (-5.5mm,0) {\scriptsize $\$$};
			\end{tikzpicture}=\sqrt{Q}\cdot\text{id}.$$
\end{prop}

Observe that $P_{n,+}^{Spin}$, respectively $P_{n,-}^{Spin}$, have a normalized trace $tr$ given by the tangles
$$tr(x)=\dfrac{1}{\sqrt{Q}^n}\begin{tikzpicture}[align1]
\clip[rounded corners] (0,-.6) rectangle (1,.6);
\draw[line width=1mm, rounded corners] (.45,-.4) rectangle (.85,.4);
\node at (.45,0)[Tbox, minimum width=.5cm, minimum height=.5cm]{\scriptsize $x$};
\node at (.85,0)[Tcirc, inner sep=0]{\scriptsize $n$};
\node at (.1,.45) {\scriptsize $\$$};
\node at (.125,0) {\scriptsize $\$$};
\draw[very thick, rounded corners] (0,-.6) rectangle (1,.6);
\end{tikzpicture}
\quad \text{ respectively }\quad
tr(x)=\dfrac{1}{\sqrt{Q}^{n+1}}\begin{tikzpicture}[align1]
\clip[rounded corners] (-.2,-.6) rectangle (1.1,.6);
\draw[shaded, rounded corners] (0,-.5) rectangle (1,.5);
\draw[line width=1mm, rounded corners, unshaded] (.45,-.4) rectangle (.85,.4);
\node at (.45,0)[Tbox, minimum width=.5cm, minimum height=.5cm]{\scriptsize $x$};
\node at (.85,0)[Tcirc, inner sep=0]{\scriptsize $n$};
\node at (-.1,.45) {\scriptsize $\$$};
\node at (.125,0) {\scriptsize $\$$};
\draw[very thick, rounded corners] (-.2,-.6) rectangle (1.1,.6);
\end{tikzpicture}$$
where some shadings have been omitted and thick lines denote $n$ parallel strings. We will use $P_{n,\pm}^{Spin}$ to denote these $C^*$-algebras and in more complicated tangles we will omit the shading.

\begin{defn}\normalfont
Define an inner product on $P_{n,+}^{Spin}$ by the tangle
$\inn{\xi}{\eta}_{Spin}=$\begin{tikzpicture}[align1]
\clip[rounded corners] (0,-.5) rectangle (2,.5);
\draw[line width=1mm] (.5,0)--(1.5,0);
\node at (.5,0)[Tbox, minimum width=0cm, minimum height=0cm]{\scriptsize $\xi$};
\node at (1.5,0)[Tbox, minimum width=0cm, minimum height=0cm]{\scriptsize $\eta^*$};
\node at (1,0)[Tcirc, inner sep=0cm]{\scriptsize $2n$};
\node at (.1,.35) {\scriptsize $\$$};
\node at (.2,0) {\scriptsize $\$$};
\node at (1.85,0) {\scriptsize $\$$};
\draw[rounded corners, very thick] (0,-.5) rectangle (2,.5);
\end{tikzpicture} called the \textit{spin inner product}.
We will also use $\norm{\cdot}_{Spin}$ to denote the norm coming from this inner product called the \textit{spin norm}.
\end{defn}

Observe that the spin inner product and spin norm are unnormalized versions of the trace inner product and $2$-norm from the trace. The spin inner product is more natural when constructing orthonormal basis in $P^{Spin}$ because it provides the right normalization for a cable cutting operation we will define later.

\begin{defn}\normalfont (\cite{Jon99},\cite{Jon19})
A $Q \times Q$ complex matrix, $H$, is called a \textit{complex Hadamard matrix} if $HH^*=QI$ and $\abs{H_{i,j}}=1$ for all $i,j$. Define $u=\sum_{i,j=1}^Q H_{i,j} \hat{i} \otimes \hat{j} \in P_{2,+}^{Spin}$ and observe that $u$ satisfies the following equalities
$$
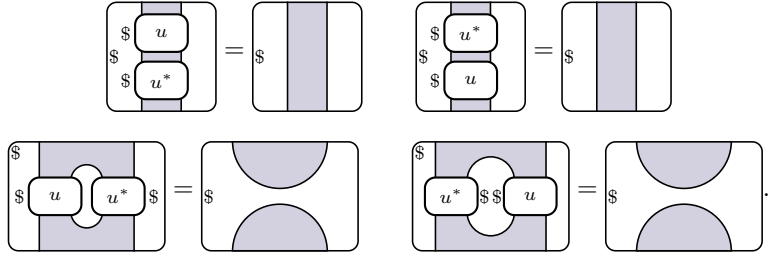
\begin{tikzpicture}[align1]
\clip[rounded corners] (.3,-.7) rectangle (1.7,.7);
\path[shaded] (.75,-2)--(1.25,-2)--(1.25,2)--(.75,2)--cycle;
\draw (.75,-2)--(.75,2)--(1.25,2)--(1.25,-2)--cycle;
\node at (1,.3) [Tbox, minimum width=.7cm, minimum height=.5cm]{\scriptsize $u$};
\node at (1,-.3) [Tbox, minimum width=.7cm, minimum height=.5cm]{\scriptsize $u^*$};
\node at (.4,0) {\scriptsize $\$$};
\node at (.55,.3) {\scriptsize $\$$};
\node at (.55,-.3) {\scriptsize $\$$};
\draw[rounded corners, very thick] (.3,-.7) rectangle (1.7,.7);
\end{tikzpicture}=
\begin{tikzpicture}[align1]
\clip[rounded corners] (.3,-.7) rectangle (1.7,.7);
\path[shaded] (.75,-2)--(1.25,-2)--(1.25,2)--(.75,2)--cycle;
\draw (.75,-2)--(.75,2)--(1.25,2)--(1.25,-2)--cycle;
\node at (.4,0) {\scriptsize $\$$};
\draw[rounded corners, very thick] (.3,-.7) rectangle (1.7,.7);
\end{tikzpicture} \quad \quad 
\begin{tikzpicture}[align1]
\clip[rounded corners] (.3,-.7) rectangle (1.7,.7);
\path[shaded] (.75,-2)--(1.25,-2)--(1.25,2)--(.75,2)--cycle;
\draw (.75,-2)--(.75,2)--(1.25,2)--(1.25,-2)--cycle;
\node at (1,.3) [Tbox, minimum width=.7cm, minimum height=.5cm]{\scriptsize $u^*$};
\node at (1,-.3) [Tbox, minimum width=.7cm, minimum height=.5cm]{\scriptsize $u$};
\node at (.4,0) {\scriptsize $\$$};
\node at (.55,.3) {\scriptsize $\$$};
\node at (.55,-.3) {\scriptsize $\$$};
\draw[rounded corners, very thick] (.3,-.7) rectangle (1.7,.7);
\end{tikzpicture}=
\begin{tikzpicture}[align1]
\clip[rounded corners] (.3,-.7) rectangle (1.7,.7);
\path[shaded] (.75,-2)--(1.25,-2)--(1.25,2)--(.75,2)--cycle;
\draw (.75,-2)--(.75,2)--(1.25,2)--(1.25,-2)--cycle;
\node at (.4,0) {\scriptsize $\$$};
\draw[rounded corners, very thick] (.3,-.7) rectangle (1.7,.7);
\end{tikzpicture}$$
$$\begin{tikzpicture}[align1]
\clip[rounded corners] (0,-.7) rectangle (2,.7);
\path[shaded] (.4,-2)--(.4,0)--(1.6,0)--(1.6,-2)--cycle;
\path[shaded] (.4,0)--(.4,2)--(1.6,2)--(1.6,0)--cycle;
\path[shadedw] (.8,.2) arc (180:0:.2) -- (1.2,-.2) arc(0:-180:.2) -- cycle;
\draw (.4,-2)--(.4,2)--(1.6,2)--(1.6,-2)--cycle;
\draw (.8,.2) arc (180:0:.2) -- (1.2,-.2) arc(0:-180:.2) -- cycle;
\node at (1.4,0) [Tbox, minimum width=.7cm, minimum height=.5cm]{\scriptsize $u^*$};
\node at (.6,0) [Tbox, minimum width=.7cm, minimum height=.5cm]{\scriptsize $u$};
\node at (.1,.55) {\scriptsize $\$$};
\node at (.15,0) {\scriptsize $\$$};
\node at (1.85,0) {\scriptsize $\$$};
\draw[rounded corners, very thick] (0,-.7) rectangle (2,.7);
\end{tikzpicture}=
\begin{tikzpicture}[align1]
\clip[rounded corners] (0,-.7) rectangle (2,.7);
\draw[shaded] (.4,.7) arc(-180:0:.6) --cycle;
\draw[shaded] (.4,-.7) arc(180:0:.6) --cycle;
\node at (.1,0) {\scriptsize $\$$};
\draw[rounded corners, very thick] (0,-.7) rectangle (2,.7);
\end{tikzpicture} \quad \quad
\begin{tikzpicture}[align1]
\clip[rounded corners] (0,-.7) rectangle (2,.7);
\path[shaded] (.3,-2)--(.3,0)--(1.7,0)--(1.7,-2)--cycle;
\path[shaded] (.3,0)--(.3,2)--(1.7,2)--(1.7,0)--cycle;
\path[shadedw] (.7,.2) arc (180:0:.3) -- (1.3,-.2) arc(0:-180:.3) -- cycle;
\draw (.3,-2)--(.3,2)--(1.7,2)--(1.7,-2)--cycle;
\draw (.7,.2) arc (180:0:.3) -- (1.3,-.2) arc(0:-180:.3) -- cycle;
\node at (1.5,0) [Tbox, minimum width=.7cm, minimum height=.5cm]{\scriptsize $u$};
\node at (.5,0) [Tbox, minimum width=.7cm, minimum height=.5cm]{\scriptsize $u^*$};
\node at (.1,.55) {\scriptsize $\$$};
\node at (.92,0) {\scriptsize $\$$};
\node at (1.08,0) {\scriptsize $\$$};
\draw[rounded corners, very thick] (0,-.7) rectangle (2,.7);
\end{tikzpicture}=
\begin{tikzpicture}[align1]
\clip[rounded corners] (0,-.7) rectangle (2,.7);
\draw[shaded] (.4,.7) arc(-180:0:.6) --cycle;
\draw[shaded] (.4,-.7) arc(180:0:.6) --cycle;
\node at (.1,0) {\scriptsize $\$$};
\draw[rounded corners, very thick] (0,-.7) rectangle (2,.7);
\end{tikzpicture}.$$
Furthermore, these conditions are equivalent to $u$ coming from a complex Hadamard matrix. Any element $u \in P_{2,+}^{Spin}$ satisfying these equalities is called a \textit{biunitary} in $P^{Spin}$.
\end{defn}

One might expect the first equality above to yield $Q \cdot id_{P_{2,+}^{Spin}}$, but the action of the planar operad on $P^{Spin}$ absorbs the factor of $Q$. We encourage the reader to verify the equalities in $P^{Spin}$. Observe that these are equivalent to type $II$ Reidemeister moves and so we will adopt notation from knot theory for $u$ and $u^*$. Let 
$$
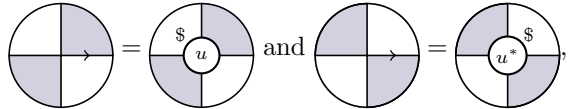
\begin{tikzpicture}[align.7]
\clip (0,0) circle (1cm);
\draw[shaded] (0,0)--(2,0)--(0,2)--(0,0);
\draw[shaded] (0,0)--(-2,0)--(0,-2)--(0,0);
\draw[->] (.5cm,0)--(.55cm,0);
\draw[very thick] (0,0) circle (1cm);
\end{tikzpicture}
=
\begin{tikzpicture}[align.7]
\clip (0,0) circle (1cm);
\draw[shaded] (0,0)--(2,0)--(0,2)--(0,0);
\draw[shaded] (0,0)--(-2,0)--(0,-2)--(0,0);
\node at (0,0)[Tcirc, inner sep=1mm]{\scriptsize $u$};
\node at (-4mm,4mm){\scriptsize $\$$};
\draw[very thick] (0,0) circle (1cm);
\end{tikzpicture}
\text{ and }
\begin{tikzpicture}[align.7]
\clip (0,0) circle (1cm);
\draw[shaded, very thick] (0,0) circle (1cm);
\draw[unshaded] (0,0)--(2,0)--(0,2)--(0,0);
\draw[unshaded] (0,0)--(-2,0)--(0,-2)--(0,0);
\draw[->] (.5cm,0)--(.55cm,0);
\draw[very thick] (0,0) circle (1cm);
\end{tikzpicture}
=
\begin{tikzpicture}[align.7]
\clip (0,0) circle (1cm);
\draw[shaded, very thick] (0,0) circle (1cm);
\draw[unshaded] (0,0)--(2,0)--(0,2)--(0,0);
\draw[unshaded] (0,0)--(-2,0)--(0,-2)--(0,0);
\node at (0,0)[Tcirc, inner sep=.5mm]{\scriptsize $u^*$};
\node at (4mm,4mm){\scriptsize $\$$};
\draw[very thick] (0,0) circle (1cm);
\end{tikzpicture},$$
then we have the type $II$ Reidemeister moves
$$
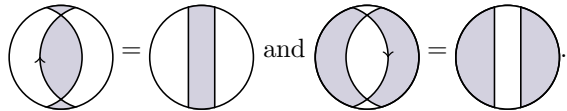
\begin{tikzpicture}[align.7, scale=.5]
\clip (0,0) circle (2cm);
\draw[shaded] (-1.2,2) arc (90:-90:2cm) --(-1.2,-2)--(1.2,-2) arc (270:90:2cm) --(-1.2,2);
\draw[->] (-.8,0)--(-.8,.1mm);
\draw[very thick] (0,0) circle (2cm);
\end{tikzpicture}
=
\begin{tikzpicture}[align.7, scale=.5]
\clip (0,0) circle (2cm);
\draw[shaded] (-.5,2)--(-.5,-2)--(.5,-2)--(.5,2)--(-.5,2);
\draw[very thick] (0,0) circle (2cm);
\end{tikzpicture}
\text{ and }
\begin{tikzpicture}[align.7, scale=.5]
\clip (0,0) circle (2cm);
\draw[shaded, very thick] (0,0) circle (2cm);
\draw[unshaded] (-1.2,2) arc (90:-90:2cm) --(-1.2,-2)--(1.2,-2) arc (270:90:2cm) --(-1.2,2);
\draw[<-] (.8,0)--(.8,.1mm);
\draw[very thick] (0,0) circle (2cm);
\end{tikzpicture}
=
\begin{tikzpicture}[align.7, scale=.5]
\clip (0,0) circle (2cm);
\draw[shaded, very thick] (0,0) circle (2cm);
\draw[unshaded] (-.5,2)--(-.5,-2)--(.5,-2)--(.5,2)--(-.5,2);
\draw[very thick] (0,0) circle (2cm);
\end{tikzpicture}.$$

We will also use the following notation to simplify the use of biunitaries. Let $
\begin{tikzpicture}[align.7]
\clip[rounded corners] (0,-.8) rectangle (2,.8);
\draw[line width=1mm] (1,-1)--(1,1);
\draw[->] (.7,0)--(.7,.4);
\node at (1,-.2)[Tcirc, inner sep=.05mm]{\scriptsize $k$};
\node at (.15,0){\scriptsize $\$$};
\draw[rounded corners,very thick] (0,-.8) rectangle (2,.8);
\end{tikzpicture}
$ denote $k$ parallel strings with the left most string being oriented up and alternating orientations from left to right. Similarly let 
$\begin{tikzpicture}[align.7]
\clip[rounded corners] (0,-.8) rectangle (2,.8);
\draw[line width=1mm] (1,-1)--(1,1);
\draw[->] (1.3,0)--(1.3,.4);
\node at (1,-.2)[Tcirc, inner sep=.05mm]{\scriptsize $k$};
\node at (.15,0){\scriptsize $\$$};
\draw[rounded corners,very thick] (0,-.8) rectangle (2,.8);
\end{tikzpicture}$ denote alternating orientations from right to left.

\begin{prop}\cite{Jon19}
Let $P_{n,\pm}^{Spin}$ denote the $C^*$-algebras defined by the multiplication tangle and conjugation. Then there are injective unital trace-preserving $*$-algebra homomorphisms  $i_n \colon P_{n,\pm}^{Spin} \to P_{n+1,\pm}^{Spin}$ defined by
$i_n(x)=$
\begin{tikzpicture}[align.7]
\clip[rounded corners] (.1,-.5) rectangle (1.8,.75);
\draw[line width=1mm] (1,-1)--(1,1);
\draw (1.5,-1)--(1.5,1);
\node at (1,0)[Tbox, inner sep=1mm]{\scriptsize $x$};
\node at (.6,0) {\scriptsize $\$$};
\node at (.25,.3) {\scriptsize $\$$};
\node at (1,.5)[Tcirc, inner sep=0]{\scriptsize $n$};
\draw[very thick, rounded corners] (.1,-.5) rectangle (1.8,.75);
\end{tikzpicture} 
and $P_{2n,+}^{Spin} \cong M_{Q^n}(\mathbb C)$, $P_{2n+1,+}^{Spin} \cong M_{Q^n}(\mathbb C) \otimes \Delta_Q$ where $\Delta_Q$ is the algebra of diagonal matrices.
Furthermore, the tower of algebras $\mathbb C = P_{0,+}^{Spin} \subset P_{1,+}^{Spin} \subset P_{2,+}^{Spin} \subset ...$ is a basic construction with Jones projections $e_{n+2} =\frac{1}{\sqrt{Q}}$
\begin{tikzpicture}[align.7]
\clip[rounded corners] (0,-.4) rectangle (1.8,.4);
\draw[line width=1mm] (.6,-1)--(.6,1);
\draw (1.2,.4) circle (.3);
\draw (1.2,-.4) circle (.3);
\node at (.15,0) {\scriptsize $\$$};
\node at (.6,0)[Tcirc, inner sep=0]{\scriptsize $n$};
\draw[very thick, rounded corners] (0,-.4) rectangle (1.8,.4);
\end{tikzpicture} and traces induced by $tr$ on the planar algebra.
\end{prop}

\begin{proof}
The Jones projections defined above clearly implement the conditional expectations for $P_{n,+}^{Spin} \subset P_{n+1,+}^{Spin}$ with respect to $tr$. We may identify the algebras $P_{n,+}^{Spin}$ by building matrix units from simple tensors $B^{\otimes n} = \set{\otimes_{i=1}^n \hat{s}_i}{\hat{s}_i \in B}$. Finally, a dimension argument then forces these inclusions to be a basic construction.
\end{proof}

\begin{prop}\label{psi}
The spin model commuting square and its basic construction are given by
$$\begin{array}{ccc}
\Delta_Q & \subset & M_Q(\mathbb C)\\
\cup & & \cup \\
\mathbb C & \subset & H \Delta_Q H^* \\
\end{array}
\cong
\begin{array}{ccc}
P_{1,+}^{Spin} & \subset & P_{2,+}^{Spin}\\
\cup & & \cup \\
\mathbb C & \subset & \psi_u(P_{1,+}^{Spin}) \\
\end{array}
\text{ and }
\left(\begin{array}{ccc}
P_{n,+}^{Spin} & \subset & P_{n+1,+}^{Spin}\\
\cup & & \cup \\
\psi_u(P_{n-1,+}^{Spin}) & \subset & \psi_u(P_{n,+}^{Spin}) \\
\end{array}\right)_{n \geq 1}$$
where $\psi_u$ is given by
$$\psi_u(x)=\begin{tikzpicture}[align1]
\clip[rounded corners] (.5,-.8) rectangle (1.5,.8);
\draw[line width=1mm] (1,1)--(1,-1);
\draw (.75,1)--(.75,.6) to [out=-90, in=90] (1.35,.2)--(1.35,-.2) to [out=-90,in=90] (.75,-.6)--(.75,-1);
\draw[<-] (1.35,.01)--(1.35,-.01);
\node at (1,0)[Tbox, minimum width=.5cm, minimum height=.5cm]{\scriptsize $x$};
\node at (1,.63)[Tcirc, inner sep=0]{\scriptsize $n$};
\node at (.69,0) {\scriptsize $\$$};
\node at (.6,.65) {\scriptsize $\$$};
\draw[rounded corners, very thick] (.5,-.8) rectangle (1.5,.8);
\end{tikzpicture} \text{ for } x \in P^{Spin}_{n,+}.$$
\end{prop}

Observe that $\psi_u$ is a trace preserving $*$-homomorphism on $\bigcup_n P^{Spin}_{n,+}$ and so $\psi_u$ extends to a trace preserving $*$-homomorphism on $M=\overline{\bigcup_n P^{Spin}_{n,+}}^{w.o.}$ which gives us the index $Q$ horizontal subfactor, $\psi_u(M)=N \subset M$, of the spin model commuting square. Similarly, the vertical subfactor $P\subset R$ is given by $\psi_{u^*}(R)=P \subset R$, $R=\overline{\bigcup_n P^{Spin}_{n,+}}^{w.o.}$. We will use $R_k$ and $M_k$ to denote the Jones towers $P \subset R \subset R_1 \subset \cdots$ and $N \subset M \subset M_1 \subset \cdots$.

\begin{rem}\cite{JS97}
Applying Ocneanu compactness to $P \subset R_k$ using the symmetric commuting square,
 $\begin{array}{ccc}
P_{1,+}^{Spin} & \subset & P_{k+1,+}^{Spin}\\
\cup & & \cup \\
\mathbb C & \subset & \psi_{u}(P_{k,+}^{Spin}) \\
\end{array}$
we find that $P' \cap R_{k-1} \cong (P_{1,+}^{Spin})' \cap \psi_{u}( P_{k,+}^{Spin})$ inside $P_{k+1,+}^{Spin}$. Observe that $(P_{1,+}^{Spin})' \cap P_{k+1,+}^{Spin}$ is given by elements of the form
\begin{tikzpicture}[align.7]
\clip[rounded corners] (0,-1) rectangle (1.5,1);
\draw[shaded] (.4,-1)--(.4,1)--(1,1)--(1,-1)--(.4,-1);
\draw[line width=1mm] (1,-1)--(1,1);
\node at (1,0)[Tbox, minimum width=.5, minimum height=.5]{\scriptsize $x$};
\node at (1,.65)[Tcirc, inner sep=.05mm]{\scriptsize $k$};
\node at (.2,0) {\scriptsize $\$$};
\node at (.55,0) {\scriptsize $\$$};
\draw[very thick, rounded corners] (0,-1) rectangle (1.5,1);
\end{tikzpicture}
for $x \in P_{k,-}^{Spin}$. Therefore $\psi_{u}(y) \in P_{k+1,+}^{Spin}$ belongs to $(P_{1,+}^{Spin})' \cap \psi_{u}( P_{k,+}^{Spin})$ iff there exists an $x \in P_{k,-}^{Spin}$ such that
\begin{tikzpicture}[align.7]
\clip[rounded corners] (0,-1) rectangle (2,1);
\draw (.35,-1)to [out=90,in=-90] (1.6,0)--(1.6,0) to [out=90,in=-90] (.35,1) -- (1,1)--(1,-1)--(.35,-1);
\draw[->] (1.6,0)--(1.6,.01);
\draw[line width=1mm]  (1,1)--(1,-1);
\node at (1,.75)[Tcirc, inner sep=.05mm]{\scriptsize $k$};
\node at (1,0)[Tbox, minimum width=.5, minimum height=.5]{\scriptsize $y$};
\node at (.15,0) {\scriptsize $\$$};
\node at (.5,0) {\scriptsize $\$$};
\draw[very thick, rounded corners] (0,-1) rectangle (2,1);
\end{tikzpicture}
$=$
\begin{tikzpicture}[align.7]
\clip[rounded corners] (0,-1) rectangle (2,1);
\draw (.35,-1)--(.35,1);
\draw[line width=1mm]  (1,1)--(1,-1);
\node at (1,.65)[Tcirc, inner sep=.05mm]{\scriptsize $k$};
\node at (1,0)[Tbox, minimum width=.5, minimum height=.5]{\scriptsize $x$};
\node at (.15,0) {\scriptsize $\$$};
\node at (.5,0) {\scriptsize $\$$};
\draw[very thick, rounded corners] (0,-1) rectangle (2,1);
\end{tikzpicture}.
\end{rem}

\begin{defn}\normalfont (\cite{Jon99},\cite{Jon19})
We will call $y \in P_{k,+}^{Spin}$ \textit{flat} with respect to $u$ if there exists $x \in P_{k,-}^{Spin}$ such that
\begin{tikzpicture}[align.7]
\clip[rounded corners] (0,-1) rectangle (2,1);
\draw (0,.65)--(2,.65);
\draw[->] (1.5,.65)--(1.4,.65);
\draw[line width=1mm]  (1,1)--(1,-1);
\node at (1,-.65)[Tcirc, inner sep=.05mm]{\scriptsize $k$};
\node at (1,0)[Tbox, minimum width=.5, minimum height=.5]{\scriptsize $y$};
\node at (.15,0) {\scriptsize $\$$};
\node at (.5,0) {\scriptsize $\$$};
\draw[very thick, rounded corners] (0,-1) rectangle (2,1);
\end{tikzpicture}
$=$
\begin{tikzpicture}[align.7]
\clip[rounded corners] (0,-1) rectangle (2,1);
\draw (0,-.5)--(2,-.5);
\draw[->] (1.5,-.5)--(1.4,-.5);
\draw[line width=1mm]  (1,1)--(1,-1);
\node at (1,.65)[Tcirc, inner sep=.05mm]{\scriptsize $k$};
\node at (1,0)[Tbox, minimum width=.5, minimum height=.5]{\scriptsize $x$};
\node at (.15,-.75) {\scriptsize $\$$};
\node at (.5,0) {\scriptsize $\$$};
\draw[very thick, rounded corners] (0,-1) rectangle (2,1);
\end{tikzpicture} where $u$ is the biunitary used for the diagrams. The flat elements with respect to $u$ clearly form a unital subalgebra of $P_{k,+}^{Spin}$ that we will denote by $P_{k,+}^{u}$.
\end{defn}

\begin{prop} (\cite{Jon99},\cite{Jon19})
$P_{0,+}^{u}$ and $P_{1,+}^{u}$ are isomorphic to $\mathbb C$ and so $P\subset R$ is an extremal subfactor. Furthermore, $P_{2,+}^{u}$ is abelian.
\end{prop}

\begin{proof}
$dim (P_{0,+}^{u}) =1$ is immediate since $dim (P_{0,+}^{Spin}) = 1$.

Let $y_q$ and $x_q$ be the $\hat{q}$ coefficients of $x,y \in \mathbb C^Q \cong P_{1,\pm}^{Spin}$. Consider the following state on the flatness condition for $x$ and $y$:
\begin{tikzpicture}[align.7]
\clip[rounded corners] (0,-1) rectangle (2,1);
\draw[shaded] (0,.5)--(2,.5)--(2,-1)--(1,-1)--(1,1)--(0,1)--(0,.5);
\draw[->] (1.5,.5)--(1.4,.5);
\node at (.5,.7) {\scriptsize $q$};
\node at (1.6,-.5) {\scriptsize $r$};
\node at (1,0)[Tbox, minimum width=.5, minimum height=.5]{\scriptsize $y$};
\node at (.15,0) {\scriptsize $\$$};
\node at (.5,0) {\scriptsize $\$$};
\draw[very thick, rounded corners] (0,-1) rectangle (2,1);
\end{tikzpicture}
$=$
\begin{tikzpicture}[align.7]
\clip[rounded corners] (0,-1) rectangle (2,1);
\draw[shaded] (0,-.5)--(2,-.5)--(2,-1)--(1,-1)--(1,1)--(0,1)--(0,-.5);
\draw[->] (1.5,-.5)--(1.4,-.5);
\node at (.4,.5) {\scriptsize $q$};
\node at (1.5,-.8) {\scriptsize $r$};
\node at (1,0)[Tbox, minimum width=.5, minimum height=.5]{\scriptsize $x$};
\node at (.15,-.75) {\scriptsize $\$$};
\node at (.5,0) {\scriptsize $\$$};
\draw[very thick, rounded corners] (0,-1) rectangle (2,1);
\end{tikzpicture}. This is equivalent to the equation $\overline{u}_{q,r}y_r = \overline{u}_{q,r}x_q$ for all $q,r \in \{\hat{1},...\hat{Q}\}$. Since entries in $u$ have modulus one, we have $y_r = x_q$ for all $q,r$ and so $y \in \mathbb C \cdot 1$.

Finally, observe that $P_{2,-}^{Spin} \cong \Delta_Q \otimes \Delta_Q$ and so $P_{2,-}^{u}$ is abelian. We also have an anti-isomorphism $\rho^3 \circ \psi_{u} \colon P_{2,+}^{u} \to P_{2,-}^{u}$ given by
$\rho^3 \circ \psi_{u^*} (y)=$
\begin{tikzpicture}[align.7]
\clip[rounded corners] (0,-1) rectangle (2,1);
\begin{scope}[rotate=180, shift=(0:-2cm)]
\draw[shaded] (.35,-1)to [out=90,in=-90] (1.6,0)--(1.6,0) to [out=90,in=-90] (.35,1) -- (.8,1)--(.8,-1)--(1.2,-1)--(1.2,1)--(2,1)--(2,-1)--(.35,-1);
\draw[->] (1.6,0)--(1.6,.01);
\node at (1,0)[Tbox, minimum width=.5, minimum height=.5]{\rotatebox{180}{\scriptsize $y$}};
\node at (.5,0) {\scriptsize $\$$};
\end{scope}
\node at (.15,0) {\scriptsize $\$$};
\draw[very thick, rounded corners] (0,-1) rectangle (2,1);
\end{tikzpicture}
$=$
\begin{tikzpicture}[align.7]
\clip[rounded corners] (0,-1) rectangle (2,1);
\begin{scope}[rotate=180, shift=(0:-2cm)]
\draw[shaded] (.35,-1)--(.35,1)--(.8,1)--(.8,-1)--(1.2,-1)--(1.2,1)--(2,1)--(2,-1)--(.35,-1);
\node at (1,0)[Tbox, minimum width=.5, minimum height=.5]{\rotatebox{180}{\scriptsize $x$}};
\node at (.5,0) {\scriptsize $\$$};
\end{scope}
\node at (.15,0) {\scriptsize $\$$};
\draw[very thick, rounded corners] (0,-1) rectangle (2,1);
\end{tikzpicture} where $\rho$ is the counter clockwise rotation tangle on $P_{2,+}^{Spin}$. Therefore $P_{2,+}^{u} \cong P' \cap R_1$ is also abelian.
\end{proof}

\begin{defn}\normalfont
Let $1 \otimes \colon P_{k,-}^{Spin} \to P_{k+1,+}^{Spin}$ denote the trace preserving $*$-algebra morphism\\
\begin{tikzpicture}[align.7]
\clip[rounded corners] (0,-1) rectangle (1.5,1);
\draw[shaded] (-1,-2)--(-1,2)--(1,2)--(1,-2)--(-1,-2);
\draw[line width=1mm] (1,-1)--(1,1);
\node at (1,0)[Tbox, minimum width=.5, minimum height=.5]{\scriptsize $x$};
\node at (1,.65)[Tcirc, inner sep=.05mm]{\scriptsize $k$};
\node at (.2,0) {\scriptsize $\$$};
\node at (.55,0) {\scriptsize $\$$};
\draw[very thick, rounded corners] (0,-1) rectangle (1.5,1);
\end{tikzpicture}
$\mapsto$
\begin{tikzpicture}[align.7]
\clip[rounded corners] (0,-1) rectangle (1.5,1);
\draw[shaded] (.4,-2)--(.4,2)--(1,2)--(1,-2)--(.4,-2);
\draw[line width=1mm] (1,-1)--(1,1);
\node at (1,0)[Tbox, minimum width=.5, minimum height=.5]{\scriptsize $x$};
\node at (1,.65)[Tcirc, inner sep=.05mm]{\scriptsize $k$};
\node at (.2,0) {\scriptsize $\$$};
\node at (.55,0) {\scriptsize $\$$};
\draw[very thick, rounded corners] (0,-1) rectangle (1.5,1);
\end{tikzpicture}.
Observe that $1\otimes$ extends to an injective trace preserving map\\
$1\otimes \colon \overline{\bigcup_k P_{k,-}^{Spin}}^{w.o.} \to M = \overline{\bigcup_k P_{k,+}^{Spin}}^{w.o.}$. Hence we may define a von Neumann subalgebra\\ $L=1\otimes \overline{\bigcup_k P_{k,-}^{Spin}}^{w.o.}$ of $M$.
\end{defn}

\begin{rem}\normalfont
We have already shown that $P' \cap R_{k-1} \cong (P_{1,+}^{Spin})' \cap \psi_{u}( P_{k,+}^{Spin}) \cong 1 \otimes P_{k,-}^{Spin} \cap \psi_{u}( P_{k,+}^{Spin})$ and so $\overline{\bigcup_k P'\cap R_k }^{w.o.} \cong N \cap L$. Thus we have a quadrilateral of von Neumann algebras $\begin{array}{ccc} N & \subset & M \\ \cup & & \cup \\ N \cap L & \subset & L \\ \end{array}$. Let $E_N$ and $E_L$ be the trace preserving conditional expectations to $N$ and $L$ respectively. In \cite{SW94}, Sano and Watatani defined the angle operator $\Theta = \sqrt{E_NE_LE_N - E_N \land E_L}$ and showed that the spectrum $\sigma(\Theta)$ is finite iff $[M : N\cap L] < \infty$ provided that $M,N,L$ and $N \cap L$ are $II_1$ factors. This gives us reason to expect the spectrum of the angle operator to contain important information about the higher relative commutants of $P \subset R$. We will use a slight variation of the angle operator that is more convenient for a planar algebra description.
\end{rem}

\begin{defn}\normalfont \cite{Jon19}
Let $N$ and $L$ be as above. Then there are unique trace preserving conditional expectations $E_N$ and $E_L$. Define $\Theta_{u} = E_NE_LE_N \in B(L^2(M))$ as the \textit{angle operator}.
\end{defn}

\begin{lem}[Cable cutting]\label{cabling}
Let $\{ b_i\}_{i=1}^{Q^n} \subset P_{n,+}^{Spin}$ be an orthonormal basis of $P_{n,+}^{Spin}$ with respect to $\inn{\cdot}{\cdot}_{Spin}$, then
$$\begin{tikzpicture}[align1]
\clip[rounded corners] (.5,-.8) rectangle (1.5,.8);
\draw[line width=1mm] (1,1)--(1,-1);
\node at (1,0)[Tcirc, inner sep=0]{\scriptsize $2n$};
\node at (.6,.65) {\scriptsize $\$$};
\draw[rounded corners, very thick] (.5,-.8) rectangle (1.5,.8);
\end{tikzpicture}
=\sum_{i=1}^{Q^n}
\begin{tikzpicture}[align1]
\clip[rounded corners] (.5,-.8) rectangle (1.5,.8);
\draw[line width=1mm] (1,.3)--(1,1);
\draw[line width=1mm] (1,-.3)--(1,-1);
\node at (1,.3)[Tbox, minimum width=.5cm,minimum height=.5cm]{\scriptsize $b_i$};
\node at (1,-.3)[Tbox, minimum width=.5cm,minimum height=.5cm]{\scriptsize $b_i^*$};
\node at (.6,.65) {\scriptsize $\$$};
\node at (.65,.3) {\scriptsize $\$$};
\node at (.65,-.3) {\scriptsize $\$$};
\draw[rounded corners, very thick] (.5,-.8) rectangle (1.5,.8);
\end{tikzpicture}.$$
\end{lem}

\begin{proof}
Observe that $x \in P_{2n,+}^{Spin} \cong M_{Q^n}(\mathbb C)$ has a faithful irreducible representation on $\xi \in P_{n,+}^{Spin}$ by
$x\xi=\begin{tikzpicture}[align1]
\clip[rounded corners] (-.6,-.4) rectangle (.9,.5);
\draw[line width=1mm] (-.6,0)--(.5,0);
\node at (0,0)[Tbox, minimum width=0cm,minimum height=0cm]{\scriptsize $x$};
\node at (.5,0)[Tbox, minimum width=0cm,minimum height=0cm]{\scriptsize $\xi$};
\node at (-.4,0)[Tcirc, inner sep=0]{\scriptsize $n$};
\node at (.75,0) {\scriptsize $\$$};
\node at (0,0.3) {\scriptsize $\$$};
\node at (-.5,.35) {\scriptsize $\$$};
\draw[rounded corners, very thick] (-.6,-.4) rectangle (.9,.5);
\end{tikzpicture}.$ Taking $\{b^*_i\}_{i=1}^{Q^n}$ as a basis for $P_{n,+}^{Spin}$, both sides of the equality above act by the identity, hence they are equal.
\end{proof}

\begin{prop}\label{cond exp}
Let $\{ b_i\}_{i=1}^{Q} \subset P_{1,+}^{Spin}$ be an orthonormal basis, then for $x \in P_{k+1,+}^{Spin}$, $E_N$ and $E_L$ are given by the following diagrams:

$$
E_N(x) = \dfrac{1}{\sqrt{Q}}
\begin{tikzpicture}[align.7]
\clip[rounded corners] (0,-1.5) rectangle (3,1.5);
\draw[line width=1mm] (1.5,-1.5)--(1.5,1.5);
\draw (1.2,.3) to [out=90,in=90] (2.3,.2)--(2.3,-.2) to [out=-90,in=-90] (1.2,-.3)--cycle;
\draw (1,-1.5)--(1,-1.2) to [out=90,in=-90] (2.6,-.3)--(2.6,.3) to [out=90,in=-90] (1,1.2)--(1,1.5);
\draw[<-] (2.3,-.01)--(2.3,.01);
\draw[->] (2.6,-.01)--(2.6,.01);
\node at (1.5,0)[Tbox, minimum width=.75cm, minimum height=.5cm]{\scriptsize $x$};
\node at (1.5,1.2)[Tcirc, inner sep=.05mm]{\scriptsize $k$};
\node at (.15,0) {\scriptsize $\$$};
\node at (.85,0) {\scriptsize $\$$};
\draw[very thick, rounded corners] (0,-1.5) rectangle (3,1.5);
\end{tikzpicture}
\,\,\,\,\,\,\,\,\,\,\,\,\,\,\,\,\,\,\,\,\,
E_L(x)=\dfrac{1}{\sqrt{Q}}
\begin{tikzpicture}[align.7]
\clip[rounded corners] (0,-1.5) rectangle (2.5,1.5);
\draw[shaded] (.4,-1.5)--(.4,1.5)--(1.5,1.5)--(1.5,-1.5) --cycle;
\draw[unshaded] (1.2,.3) arc (0:180:.25cm) --(.7,-.3) arc (180:360:.25cm) --cycle;
\draw[line width=1mm] (1.5,-1.5)--(1.5,1.5);
\node at (1.4,0)[Tbox, minimum width=.5cm, minimum height=.5cm]{\scriptsize $x$};
\node at (1.5,1)[Tcirc, inner sep=.05mm]{\scriptsize $k$};
\node at (.15,0) {\scriptsize $\$$};
\node at (.9,0) {\scriptsize $\$$};
\draw[very thick, rounded corners] (0,-1.5) rectangle (2.5,1.5);
\end{tikzpicture}
=\dfrac{1}{\sqrt{Q}}
\sum_{i=1}^Q
\begin{tikzpicture}[align.7]
\clip[rounded corners] (0,-1.5) rectangle (2.5,1.5);
\draw[shaded] (.9,-1.5)--(.9,1.5)--(1.7,1.5)--(1.7,-1.5) --cycle;
\draw[line width=1mm] (1.7,-1.5)--(1.7,1.5);
\node at (1.35,0)[Tbox, minimum width=.8cm, minimum height=.5cm]{\scriptsize $x$};
\node at (.9,.9)[Tbox, minimum width=.2cm, minimum height=.2cm]{\scriptsize $b_i$};
\node at (.9,-.9)[Tbox, minimum width=.2cm, minimum height=.2cm]{\scriptsize $b_i^*$};
\node at (1.7,1)[Tcirc, inner sep=.05mm]{\scriptsize $k$};
\node at (.15,0) {\scriptsize $\$$};
\node at (.6,0) {\scriptsize $\$$};
\node at (.4,.9) {\scriptsize $\$$};
\node at (.4,-.9) {\scriptsize $\$$};
\draw[very thick, rounded corners] (0,-1.5) rectangle (2.5,1.5);
\end{tikzpicture}.$$
Furthermore, the angle operator is given by the planar tangle
$$\Theta_u(x)=\dfrac{1}{\sqrt{Q}^3}
\begin{tikzpicture}[align1]
\clip[rounded corners] (.5,1.4) rectangle (2.5,-1.4);
\draw[line width=1mm] (1.5,-1.5)--(1.5,1.5);
\draw (1,1.5)--(1,1.4) to [out=-90, in=90] (2.4,.5)--(2.4,-.5) to [out=-90,in=90] (1,-1.4)--(1,-1.5);
\draw (1.25,.25) to [out=90,in=90] (1.8,.2)--(1.8,-.2) to [out=-90,in=-90] (1.25,-.25)--cycle;
\draw (.9,.3) to [out=90,in=90] (2,.3)--(2,-.3) to [out=-90,in=-90] (.9,-.3)--cycle;
\draw (.7,.4) to [out=90,in=90] (2.2,.4)--(2.2,-.4) to [out=-90,in=-90] (.7,-.4)--cycle;
\draw[<-] (1.8,-.01)--(1.8,.01);
\draw[->] (2,-.01)--(2,.01);
\draw[<-] (2.2,-.01)--(2.2,.01);
\draw[->] (2.4,-.01)--(2.4,.01);
\node at (1.4,0)[Tbox, minimum width=.5cm, minimum height=.5cm]{\scriptsize $x$};
\node at (1.5,1.2)[Tcirc, inner sep=0.05mm]{\scriptsize $k$};
\node at (1.05,0) {\scriptsize $\$$};
\node at (.6,1.25) {\scriptsize $\$$};
\draw[rounded corners, very thick] (.5,1.4) rectangle (2.5,-1.4);
\end{tikzpicture}
\quad \quad \quad
\Theta_u(\psi_u(x))=\dfrac{1}{Q}\psi_u
\left(
\begin{tikzpicture}[align.7]
\clip[rounded corners] (0,-1.5) rectangle (3,1.5);
\draw[line width=1mm] (1.5,-1.5)--(1.5,1.5);
\draw (1.5,0) circle (.9cm);
\draw (1.5,0) circle (.7cm);
\draw[->] (2.2,0)--(2.2,0.01);
\draw[<-] (2.4,0)--(2.4,0.01);
\node at (1.5,0)[Tbox, minimum width=.5, minimum height=.5]{\scriptsize $x$};
\node at (1.5,1.2)[Tcirc, inner sep=.05mm]{\scriptsize $k$};
\node at (1,0) {\scriptsize $\$$};
\node at (.15,0) {\scriptsize $\$$};
\draw[very thick, rounded corners] (0,-1.5) rectangle (3,1.5);
\end{tikzpicture}
\right).$$
\end{prop}

\begin{proof}
First, $tr(xy)=tr(E_N(x)y)$ for all $y \in \psi_u(\bigcup_k P_{k,+}^{Spin})$ follows from type $II$ Reidemeister moves and the loop parameters of $P^{Spin}$. Similarly,  $tr(xy)=tr(E_L(x)y)$ for all $y \in 1 \otimes \bigcup_k P_{k,-}^{Spin}$ can be shown from the cable cutting lemma \ref{cabling} and the loop parameters of $P^{Spin}$. Finally, the angle operator tangle follows from the tangles for $E_N$ and $E_L$.
\end{proof}

\begin{defn}\normalfont
For a biunitary $u$ define the operator $\theta_u \colon \bigcup_k P_{k,+}^{Spin} \to \bigcup_k P_{k,+}^{Spin}$ by the tangle
$$\theta_u(x)=\dfrac{1}{Q}\begin{tikzpicture}[align.7]
\clip[rounded corners] (0,-1.5) rectangle (3,1.5);
\draw[line width=1mm] (1.5,-1.5)--(1.5,1.5);
\draw (1.5,0) circle (.9cm);
\draw (1.5,0) circle (.7cm);
\draw[->] (2.2,0)--(2.2,0.01);
\draw[<-] (2.4,0)--(2.4,0.01);
\node at (1.5,0)[Tbox, minimum width=.5, minimum height=.5]{\scriptsize $x$};
\node at (1.5,1.2)[Tcirc, inner sep=.05mm]{\scriptsize $k$};
\node at (1,0) {\scriptsize $\$$};
\node at (.15,0) {\scriptsize $\$$};
\draw[very thick, rounded corners] (0,-1.5) rectangle (3,1.5);
\end{tikzpicture}.$$
Since $\psi_u \colon L^2(N) \to L^2(M)$ is an isometry and $\Theta_u \psi_u = \psi_u \theta_u$, $\theta_u$ defines a bounded operator on $\overline{\bigcup_k P_{k,+}^{Spin}}^{\norm{\cdot}_{2,tr}}$ which we will identify as $L^2(N)$.
\end{defn}

Applying the algebra isomorphism $Ad_{\frac{1}{\sqrt{Q}}H^*}$ yields
$$\begin{array}{ccc}
\Delta_Q & \subset & M_Q(\mathbb C)\\
\cup & & \cup \\
\mathbb C & \subset & H \Delta_Q H^* \\
\end{array}
\cong
\begin{array}{ccc}
\psi_{u^*}(P_{1,+}^{Spin}) & \subset & P_{2,+}^{Spin}\\
\cup & & \cup \\
\mathbb C & \subset & P_{1,+}^{Spin} \\
\end{array}$$
and so $\Theta_{u^*}$ is the angle operator corresponding the vertical subfactor. Since
$$\begin{tikzpicture}[align.7]
\begin{scope}
\clip (0,0) circle (1cm);
\draw[shaded] (0,0)--(2,0)--(0,2)--(0,0);
\draw[shaded] (0,0)--(-2,0)--(0,-2)--(0,0);
\draw[->] (.5cm,0)--(.55cm,0);
\draw[very thick] (0,0) circle (1cm);
\end{scope}
\node at (0,-1.2) {\scriptsize $u \text{ used for crossing}$};
\end{tikzpicture}
=
\begin{tikzpicture}[align.7]
\clip (0,0) circle (1cm);
\draw[shaded] (0,0)--(2,0)--(0,2)--(0,0);
\draw[shaded] (0,0)--(-2,0)--(0,-2)--(0,0);
\node at (0,0)[Tcirc, inner sep=1mm]{\scriptsize $u$};
\node at (-4mm,4mm){\scriptsize $\$$};
\draw[very thick] (0,0) circle (1cm);
\end{tikzpicture}
=
\begin{tikzpicture}[align.7]
\begin{scope}
\clip (0,0) circle (1cm);
\draw[shaded] (0,0)--(2,0)--(0,2)--(0,0);
\draw[shaded] (0,0)--(-2,0)--(0,-2)--(0,0);
\draw[->] (0,-0.55cm)--(0,-0.54cm);
\draw[very thick] (0,0) circle (1cm);
\end{scope}
\node at (0,-1.2) {\scriptsize $u^* \text{ used for crossing}$};
\end{tikzpicture}$$

$$\begin{tikzpicture}[align.7]
\begin{scope}
\clip (0,0) circle (1cm);
\draw[shaded, very thick] (0,0) circle (1cm);
\draw[unshaded] (0,0)--(2,0)--(0,2)--(0,0);
\draw[unshaded] (0,0)--(-2,0)--(0,-2)--(0,0);
\draw[->] (.5cm,0)--(.55cm,0);
\draw[very thick] (0,0) circle (1cm);
\end{scope}
\node at (0,-1.2) {\scriptsize $u \text{ used for crossing}$};
\end{tikzpicture}
=
\begin{tikzpicture}[align.7]
\clip (0,0) circle (1cm);
\draw[shaded, very thick] (0,0) circle (1cm);
\draw[unshaded] (0,0)--(2,0)--(0,2)--(0,0);
\draw[unshaded] (0,0)--(-2,0)--(0,-2)--(0,0);
\node at (0,0)[Tcirc, inner sep=.5mm]{\scriptsize $u^*$};
\node at (4mm,4mm){\scriptsize $\$$};
\draw[very thick] (0,0) circle (1cm);
\end{tikzpicture}
=
\begin{tikzpicture}[align.7]
\begin{scope}
\clip (0,0) circle (1cm);
\draw[shaded, very thick] (0,0) circle (1cm);
\draw[unshaded] (0,0)--(2,0)--(0,2)--(0,0);
\draw[unshaded] (0,0)--(-2,0)--(0,-2)--(0,0);
\draw[->] (0,.55cm)--(0,.54cm);
\draw[very thick] (0,0) circle (1cm);
\end{scope}
\node at (0,-1.2) {\scriptsize $u^* \text{ used for crossing}$};
\end{tikzpicture}$$
we can express $\Theta_{u^*}$ with tangles where $u$ is used to interpret crossings. Here we have these tangles where $u$ is used for the crossings and the middle $k$ strings have alternating orientations
$$\Theta_{u^*}(x)=\dfrac{1}{\sqrt{Q}^3}
\begin{tikzpicture}[align1]
\clip[rounded corners] (.5,1.4) rectangle (2.5,-1.4);
\draw[line width=1mm] (1.5,-1.5)--(1.5,1.5);
\draw (1,1.5)--(1,1.4) to [out=-90, in=90] (2.4,.5)--(2.4,-.5) to [out=-90,in=90] (1,-1.4)--(1,-1.5);
\draw (1.25,.25) to [out=90,in=90] (1.8,.2)--(1.8,-.2) to [out=-90,in=-90] (1.25,-.25)--cycle;
\draw (.9,.3) to [out=90,in=90] (2,.3)--(2,-.3) to [out=-90,in=-90] (.9,-.3)--cycle;
\draw (.7,.4) to [out=90,in=90] (2.2,.4)--(2.2,-.4) to [out=-90,in=-90] (.7,-.4)--cycle;
\draw[->] (1.35,.42)--(1.35,.54);
\draw[<-] (1.35,-.42)--(1.35,-.54);
\node at (1.4,0)[Tbox, minimum width=.5cm, minimum height=.5cm]{\scriptsize $x$};
\node at (1.5,1.2)[Tcirc, inner sep=0.05mm]{\scriptsize $k$};
\node at (1.05,0) {\scriptsize $\$$};
\node at (.6,1.25) {\scriptsize $\$$};
\draw[rounded corners, very thick] (.5,1.4) rectangle (2.5,-1.4);
\end{tikzpicture}
\quad \quad \quad
\Theta_{u^*}(\psi_{u^*}(x))=\dfrac{1}{Q}\psi_{u^*}
\left(
\begin{tikzpicture}[align.7]
\clip[rounded corners] (0,-1.5) rectangle (3,1.5);
\draw[line width=1mm] (1.5,-1.5)--(1.5,1.5);
\draw (1.5,0) circle (.9cm);
\draw (1.5,0) circle (.7cm);
\draw[->] (1.32,.37)--(1.32,.56);
\draw[<-] (1.32,-.37)--(1.32,-.56);
\node at (1.5,0)[Tbox, minimum width=.5, minimum height=.5]{\scriptsize $x$};
\node at (1.5,1.2)[Tcirc, inner sep=.05mm]{\scriptsize $k$};
\node at (1,0) {\scriptsize $\$$};
\node at (.15,0) {\scriptsize $\$$};
\draw[very thick, rounded corners] (0,-1.5) rectangle (3,1.5);
\end{tikzpicture}
\right)=\dfrac{1}{Q}\psi_{u^*}(\theta_{u^*}(x)).$$

\begin{prop}\label{equalspectra}
$\sigma(\theta_u|_{P_{k,+}^{Spin}})=\sigma(\theta_{\overline{u}}|_{P_{k,+}^{Spin}})=\sigma(\theta_{u^T}|_{P_{k,+}^{Spin}})=\sigma(\theta_{u^*}|_{P_{k,+}^{Spin}})$ and $dim(P_{k,+}^{u})=dim(P_{k,+}^{\overline{u}})=dim(P_{k,+}^{u^T})=dim(P_{k,+}^{u^*})$ for all $k \in \mathbb N$.
\end{prop}

\begin{proof}
Fix $k \in \mathbb N$ and define a conjugation on $P_{k,+}^{Spin}$ by $\overline{\xi}=\sum_{\hat{b} \in B^{\otimes k}} \overline{\xi_b}\hat{b}$ where $\xi=\sum_{\hat{b} \in B^{\otimes k}} \xi_b \hat{b}$, $\xi_b \in \mathbb C$ and $B$ is the basis used to define the action of tangles on $P^{Spin}$. Observe that $\overline{\theta_u(\xi)}=\theta_{\overline{u}}(\overline{\xi})$ and so the eigenvalues of $\theta_u|_{P_{k,+}^{Spin}}$ and $\theta_{\overline{u}}|_{P_{k,+}^{Spin}}$ coincide and the eigenspaces are isomorphic by conjugation.

Define
$$\phi_u(\xi)=\begin{tikzpicture}[align.7]
	\clip[rounded corners] (0,-1.5) rectangle (3,1.5);
	\draw (1.3,0)--(1.3,.35) to [out=90, in=90] (.85,.35)--(.85,-1.5);
	\draw (1.7,0)--(1.7,-.35) to [out=-90, in=-90] (2.1,-.35)--(2.1,1.5);
	\draw[line width=1mm] (1.6,0)--(1.6,1.5);
	\draw[line width=1mm] (1.4,0)--(1.4,-1.5);
	\draw (1.5,0) circle (.9cm);
	\draw[->] (2.4,-.001)--(2.4,.001);
	\node at (1.5,0)[Tbox, minimum width=.5cm, minimum height=.5cm]{\scriptsize $\xi$};
	\node at (1,0) {\scriptsize $\$$};
	\node at (.1,0) {\scriptsize $\$$};
	\draw[very thick, rounded corners] (0,-1.5) rectangle (3,1.5);
\end{tikzpicture}
\quad\quad
\phi_u^*(\eta)=\begin{tikzpicture}[align.7]
	\clip[rounded corners] (0,-1.5) rectangle (3,1.5);
	\draw (1.3,0)--(1.3,-.35) to [out=-90, in=-90] (.85,-.35)--(.85,1.5);
	\draw (1.7,0)--(1.7,.35) to [out=90, in=90] (2.1,.35)--(2.1,-1.5);
	\draw[line width=1mm] (1.6,0)--(1.6,-1.5);
	\draw[line width=1mm] (1.4,0)--(1.4,1.5);
	\draw (1.5,0) circle (.9cm);
	\draw[<-] (2.4,-.001)--(2.4,.001);
	\node at (1.5,0)[Tbox, minimum width=.5cm, minimum height=.5cm]{\scriptsize $\eta$};
	\node at (1,0) {\scriptsize $\$$};
	\node at (.1,0) {\scriptsize $\$$};
	\draw[very thick, rounded corners] (0,-1.5) rectangle (3,1.5);
\end{tikzpicture}.$$
Interpreting string crossings with $u^T$ corresponds to reversing all orientations and so
$$\theta_{u^T}(\eta) =\dfrac{1}{Q} \begin{tikzpicture}[align.7]
	\begin{scope}
	\clip[rounded corners] (0,-1.5) rectangle (3,1.5);
	\draw[line width=1mm] (1.5,-1.5)--(1.5,1.5);
	\draw (1.5,0) circle (.9cm);
	\draw (1.5,0) circle (.7cm);
	\draw[<-] (2.2,0)--(2.2,0.01);
	\draw[->] (2.4,0)--(2.4,0.01);
	\node at (1.5,0)[Tbox, minimum width=.5cm, minimum height=.5cm]{\scriptsize $\eta$};
	\node at (1.5,1.2)[Tcirc, inner sep=.05mm]{\scriptsize $k$};
	\node at (1,0) {\scriptsize $\$$};
	\node at (.15,0) {\scriptsize $\$$};
	\draw[very thick, rounded corners] (0,-1.5) rectangle (3,1.5);
	\end{scope}
	\node at (1.5,-1.7) {\scriptsize $u \text{ used for crossings}$};
\end{tikzpicture}.$$
If $\theta_u(\xi)=\lambda \xi$ and $\lambda \neq 0$ then $\theta_{u^T}(\phi_u(\xi))=\phi_u(\theta_u(\xi)) = \lambda \phi_u(\xi)$. Furthermore, $\phi_u(\xi) \neq 0$ since $\phi_u^*(\phi_u(\xi))=Q\theta_u(\xi)=Q\lambda\xi \neq 0$. This implies that $\sigma(\theta_u|_{P_{k,+}^{Spin}})$ and $\sigma(\theta_{u^T}|_{P_{k,+}^{Spin}})$ coincide and there is a bijection between eigenspaces. Therefore $\sigma(\theta_u|_{P_{k,+}^{Spin}})=\sigma(\theta_{\overline{u}}|_{P_{k,+}^{Spin}})=\sigma(\theta_{u^T}|_{P_{k,+}^{Spin}})=\sigma(\theta_{u^*}|_{P_{k,+}^{Spin}})$ and $dim(P_{k,+}^{u})=dim(P_{k,+}^{\overline{u}})=dim(P_{k,+}^{u^T})=dim(P_{k,+}^{u^*})$.
\end{proof}

\section{Computation of $\tau(\Theta_u)$}
In \cite{Pop94} Popa constructed the symmetric enveloping algebra, $M \underset{e_N}{\boxtimes} M^{op}$, for extremal subfactors from $C^*(M,e_N,JMJ)$, the $C^*$-algebra generated by $M$, $e_N$, and $JMJ$ on $L^2(M)$ where $J$ as usual denotes the anti-linear isometry on $L^2(M)$ given by $J(x)=x^*$, $x \in M \subset L^2(M)$.  Popa's construction of the symmetric enveloping algebra and the existence of the trace $\tau$ on $C^*(M,e_N,JMJ)$ play a key role in this section. We will show that the spectra of $\theta_u$ and $\Gamma\Gamma^*$ coincide when $N \subset M$ is an amenable subfactor in the sense of \cite{Pop94b}, where $\Gamma$ denotes the principal graph of $N \subset M$ written as a $V(\Gamma_{even}) \times V(\Gamma_{odd})$ matrix. Note that in \cite{KS99}, Kodiyalam and Sunder showed that
$\left[\begin{array}{cc} 0 & \Gamma \\ \Gamma^* & 0 \\ \end{array}\right]$ and $\left[\begin{array}{cc} 0 & \Lambda \\ \Lambda^* & 0 \\ \end{array}\right]$
have the same spectrum with zero as the only possible exception where $\Lambda$ is the dual graph.

\begin{thm} \cite{Pop94}
An extremal hyperfinite subfactor $N \subset M$ is amenable iff $\norm{\Gamma}^2 = [M:N]$ where $\Gamma$ is the principal graph of $N \subset M$.
\end{thm}

There are several equivalent notions of amenability for subfactors or their standard invariants due to Popa (see \cite{Pop94},\cite{Pop94b}).

\begin{thm}\cite{Pop99}
There is a unique tracial state, $\tau$, on $C^*(M,e_N,JMJ)$ and the corresponding trace ideal, $\mathcal I_\tau$, is the unique maximal ideal of $C^*(M,e_N,JMJ)$. Furthermore, $N \subset M$ is amenable iff $C^*(M,e_N,JMJ)$ is simple.
\end{thm}

We will first define a collection of operators in $C^*(M,e_N,JMJ)$ using the planar algebra formalism. Then we will use this formalism to compute the value of $\tau$ on these operators. These operators will be defined on the dense subspace of $L^2(M)$ given by $\bigcup_k P_{k,+}^{Spin}$ then extended uniquely by continuity. Vectors $\xi \in \bigcup_k P_{k,+}^{Spin}$ will be arranged with their marked intervals on the right and strings at the left. Then the actions of $x\in P_{n,+}^{Spin} \subset M$, $Jy^*J \in JP_{n,+}^{Spin}J \subset JMJ$, and $e_N$ on $\xi \in P_{n+k,+}^{Spin}$ are given by
$$	x\xi= \begin{tikzpicture}[align1,scale=.7]
		\clip[rounded corners] (-.5,-1.25) rectangle (2.2,1.25);
		\draw[line width=1mm] (-.5,.6)--(1.5,.6);
		\draw[line width=1mm] (-.5,.2)--(1.5,.2);
		\draw[line width=1mm] (-.5,-.2)--(1.5,-.2);
		\draw[line width=1mm] (-.5,-.6)--(1.5,-.6);
		\node at (1.5,0)[Tbox, minimum width=.5, minimum height=1.2cm]{\scriptsize $\xi$};
		\node at (0,.6)[Tbox, minimum width=.2, minimum height=.2]{\scriptsize $x$};
		\node at (-.35,1.05) {\scriptsize $\$$};
		\node at (1.9,0) {\scriptsize $\$$};
		\node at (0,1) {\scriptsize $\$$};
		\node at (.7,.6)[Tcirc, inner sep=0.05mm] {\scriptsize $n$};
		\node at (.7,.2)[Tcirc, inner sep=0.05mm] {\scriptsize $k$};
		\node at (.7,-.2)[Tcirc, inner sep=0.05mm] {\scriptsize $k$};
		\node at (.7,-.6)[Tcirc, inner sep=0.05mm] {\scriptsize $n$};
		\draw[very thick,rounded corners] (-.5,-1.25) rectangle (2.2,1.25);
	\end{tikzpicture}
\quad \quad \quad
Jy^*J\xi= \begin{tikzpicture}[align1,scale=.7]
		\clip[rounded corners] (-.5,-1.25) rectangle (2.2,1.25);
		\draw[line width=1mm] (-.5,.6)--(1.5,.6);
		\draw[line width=1mm] (-.5,.2)--(1.5,.2);
		\draw[line width=1mm] (-.5,-.2)--(1.5,-.2);
		\draw[line width=1mm] (-.5,-.6)--(1.5,-.6);
		\node at (1.5,0)[Tbox, minimum width=.5, minimum height=1.2cm]{\scriptsize $\xi$};
		\node at (0,-.6)[Tbox, minimum width=.2, minimum height=.2]{\scriptsize $y$};
		\node at (-.35,1.05) {\scriptsize $\$$};
		\node at (1.9,0) {\scriptsize $\$$};
		\node at (0,-1.05) {\scriptsize $\$$};
		\node at (.7,.6)[Tcirc, inner sep=0.05mm] {\scriptsize $n$};
		\node at (.7,.2)[Tcirc, inner sep=0.05mm] {\scriptsize $k$};
		\node at (.7,-.2)[Tcirc, inner sep=0.05mm] {\scriptsize $k$};
		\node at (.7,-.6)[Tcirc, inner sep=0.05mm] {\scriptsize $n$};
		\draw[very thick,rounded corners] (-.5,-1.25) rectangle (2.2,1.25);
	\end{tikzpicture}
\quad \quad \quad
e_N \xi =\dfrac{1}{\sqrt{Q}} \begin{tikzpicture}[align1,scale=.7]
		\clip[rounded corners] (-.5,-1.1) rectangle (1.7,1.1);
		\draw[line width=1mm] (-.5,0)--(1,0);
		\draw (-.5,.5) to [out=0, in=90] (0,.05)--(0,-.05) to [out=-90,in=0] (-.5,-.5);
		\draw (.8,.5) to [out=180, in=90] (.3,.05)--(.3,-.05) to [out=-90,in=180] (.8,-.5);
		\draw[->] (-.31,-.465)--(-.311,-.4655);
		\draw[->] (.52,.425)--(.521,.4255);
		\node at (1,0)[Tbox, minimum width=.5, minimum height=1cm]{\scriptsize $\xi$};
		\node at (-.35,.9) {\scriptsize $\$$};
		\node at (1.4,0) {\scriptsize $\$$};
		\draw[very thick,rounded corners] (-.5,-1.1) rectangle (1.7,1.1);
	\end{tikzpicture}.$$

\begin{defn}\normalfont
Fix $n$, let $x,y \in P_{n,+}^{Spin}$ and define the linear operator $\pi_{x,y} \colon \bigcup_k P_{k,+}^{Spin} \to \bigcup_k P_{k,+}^{Spin}$

$$\pi_{x,y}\xi =
\begin{tikzpicture}[align1]
\clip[rounded corners] (1.2,-1.2) rectangle (3.1,1.2);

\draw[line width=1mm] (1.8,.6)--(1.8,-.6);
\draw[line width=1mm] (0,0)--(2.6,0);

\draw (1.2,.4) to [out=0, in=90] (1.6,.05)--(1.6,-.05) to [out=-90,in=0] (1.2,-.4);
\draw (2.4,.4) to [out=180, in=90] (2,.05)--(2,-.05) to [out=-90,in=180] (2.4,-.4);

\draw[->] (1.68,.15)--(1.68,.35);
\draw[->] (1.39,-.355)--(1.389,-.3555);
\draw[->] (2.21,.355)--(2.211,.3555);

\node at (2.95,0) {\scriptsize $\$$};
\node at (1.3,1.05) {\scriptsize $\$$};
\node at (1.8,1) {\scriptsize $\$$};
\node at (1.8,-1) {\scriptsize $\$$};
\node at (2.6,0)[Tbox, minimum width=.5cm, minimum height=1.3cm]{\scriptsize $\xi$};
\node at (1.8,.65)[Tbox, minimum height=.5cm, minimum width=.5cm]{\scriptsize $x$};
\node at (1.8,-.65)[Tbox, minimum height=.5cm, minimum width=.5cm]{\scriptsize $y^*$};
\node at (1.8,-.225)[Tcirc, inner sep=0cm]{\tiny $2n$};
\node at (1.39,0)[Tcirc, inner sep=0cm]{\tiny $2k$};
\draw[rounded corners, very thick] (1.2,-1.2) rectangle (3.1,1.2);
\end{tikzpicture} \quad \text{ for } \xi \in P^{Spin}_{k+1,+}.$$
This is well defined since $\pi_{x,y}$ commutes with the inclusion maps $i_{k}(\xi)=$
\begin{tikzpicture}[align.7]
\clip[rounded corners] (.1,-.75) rectangle (1.8,.75);
\draw[line width=1mm] (.1,.3)--(1.25,.3)--(1.25,-.3)--(.1,-.3);
\draw (.1,.1)--(.3,.1) arc(90:-90:.1cm)--(.1,-.1);
\node at (1.2,0)[Tbox, minimum height=.7cm]{\scriptsize $\xi$};
\node at (1.6,0) {\scriptsize $\$$};
\node at (1.6,.55) {\scriptsize $\$$};
\node at (.6,.3)[Tcirc, inner sep=0]{\scriptsize $k$};
\node at (.6,-.3)[Tcirc, inner sep=0]{\scriptsize $k$};
\draw[very thick, rounded corners] (.1,-.75) rectangle (1.8,.75);
\end{tikzpicture} by type $II$ Reidemeister moves.
\end{defn}

\begin{prop}\label{product}
The linear operators $\pi_{x,y}$ for $x,y \in P_{n,+}^{Spin}$ extend uniquely to bounded operators on $L^2(M)$ also denoted by $\pi_{x,y}$. Furthermore, these bounded operators belong to $C^*(M,e_N,JMJ)$.
\end{prop}

\begin{proof}
Let $\{b_i\}_{i=1}^Q$ be an orthonormal basis of $P_{1,+}^{Spin}$. Then
$$Qe_Nb_i Jb_jJ e_N \xi = \pi_{b_i,b_j}\xi =
\begin{tikzpicture}[align1]
\clip[rounded corners] (1.2,-1.2) rectangle (3.1,1.2);

\draw (1.7,.6)--(1.7,-.6);
\draw (1.9,.6)--(1.9,-.6);
\draw[line width=1mm] (0,0)--(2.6,0);

\draw (1.2,.4) to [out=0, in=90] (1.6,.05)--(1.6,-.05) to [out=-90,in=0] (1.2,-.4);
\draw (2.4,.4) to [out=180, in=90] (2,.05)--(2,-.05) to [out=-90,in=180] (2.4,-.4);

\draw[->] (1.7,.25)--(1.7,.251);
\draw[->] (1.9,-.25)--(1.9,-.251);
\draw[->] (1.39,-.355)--(1.389,-.3555);
\draw[->] (2.21,.355)--(2.211,.3555);

\node at (2.95,0) {\scriptsize $\$$};
\node at (1.3,1.05) {\scriptsize $\$$};
\node at (1.8,1) {\scriptsize $\$$};
\node at (1.8,-1) {\scriptsize $\$$};
\node at (2.6,0)[Tbox, minimum width=.5cm, minimum height=1.3cm]{\scriptsize $\xi$};
\node at (1.8,.65)[Tbox, minimum height=.5cm, minimum width=.5cm]{\scriptsize $b_i$};
\node at (1.8,-.65)[Tbox, minimum height=.5cm, minimum width=.5cm]{\scriptsize $b_j^*$};
\node at (1.39,0)[Tcirc, inner sep=0cm]{\tiny $2k$};
\draw[rounded corners, very thick] (1.2,-1.2) rectangle (3.1,1.2);
\end{tikzpicture} \quad \text{ for } \xi \in P^{Spin}_{k+1,+}$$
and so $\pi_{b_i,b_j} \in C^*(M,e_N,JMJ)$. Let $i_1,i_2,...,i_n$ and $j_1,j_2,...,j_n$ be indices taking values in $\{1,2,...,Q\}$ then
$$Q^{\frac{n-1}{2}} \pi_{b_{i_1},b_{j_1}}\pi_{b_{i_2},b_{j_2}} \cdots \pi_{b_{i_n},b_{j_n}} \xi = \pi_{x,y} \xi =
\begin{tikzpicture}[align1]
\clip[rounded corners] (1.2,-1.2) rectangle (3.1,1.2);

\draw[line width=1mm] (1.8,.6)--(1.8,-.6);
\draw[line width=1mm] (0,0)--(2.6,0);

\draw (1.2,.4) to [out=0, in=90] (1.6,.05)--(1.6,-.05) to [out=-90,in=0] (1.2,-.4);
\draw (2.4,.4) to [out=180, in=90] (2,.05)--(2,-.05) to [out=-90,in=180] (2.4,-.4);

\draw[->] (1.68,.15)--(1.68,.35);
\draw[->] (1.39,-.355)--(1.389,-.3555);
\draw[->] (2.21,.355)--(2.211,.3555);

\node at (2.95,0) {\scriptsize $\$$};
\node at (1.3,1.05) {\scriptsize $\$$};
\node at (1.8,1) {\scriptsize $\$$};
\node at (1.8,-1) {\scriptsize $\$$};
\node at (2.6,0)[Tbox, minimum width=.5cm, minimum height=1.3cm]{\scriptsize $\xi$};
\node at (1.8,.65)[Tbox, minimum height=.5cm, minimum width=.5cm]{\scriptsize $x$};
\node at (1.8,-.65)[Tbox, minimum height=.5cm, minimum width=.5cm]{\scriptsize $y^*$};
\node at (1.8,-.225)[Tcirc, inner sep=0cm]{\tiny $2n$};
\node at (1.39,0)[Tcirc, inner sep=0cm]{\tiny $2k$};
\draw[rounded corners, very thick] (1.2,-1.2) rectangle (3.1,1.2);
\end{tikzpicture} \quad \text{ for } \xi \in P^{Spin}_{k+1,+}$$
$$\text{where } x=
\begin{tikzpicture}[align1]
\clip[rounded corners] (0,-.5) rectangle (2.8,.6);
\draw[shaded] (.3,0) rectangle (.5,-.5);
\draw[shaded] (1.1,0) rectangle (1.3,-.5);
\draw[shaded] (2.3,0) rectangle (2.5,-.5);
\node at (.4,0)[Tbox, minimum height=.5cm, minimum width=.5cm]{\scriptsize $b_{i_1}$};
\node at (1.2,0)[Tbox, minimum height=.5cm, minimum width=.5cm]{\scriptsize $b_{i_2}$};
\node at (1.8,0) {\scriptsize $\cdots$};
\node at (2.4,0)[Tbox, minimum height=.5cm, minimum width=.5cm]{\scriptsize $b_{i_n}$};
\node at (.1,.45) {\scriptsize $\$$};
\node at (.4,.33) {\scriptsize $\$$};
\node at (1.2,.33) {\scriptsize $\$$};
\node at (2.4,.33) {\scriptsize $\$$};
\draw[rounded corners, very thick] (0,-.5) rectangle (2.8,.6);
\end{tikzpicture}
\text{ and }
y=\begin{tikzpicture}[align1]
\clip[rounded corners] (0,-.5) rectangle (2.8,.6);
\draw[shaded] (.3,0) rectangle (.5,-.5);
\draw[shaded] (1.1,0) rectangle (1.3,-.5);
\draw[shaded] (2.3,0) rectangle (2.5,-.5);
\node at (.4,0)[Tbox, minimum height=.5cm, minimum width=.5cm]{\scriptsize $b_{j_1}$};
\node at (1.2,0)[Tbox, minimum height=.5cm, minimum width=.5cm]{\scriptsize $b_{j_2}$};
\node at (1.8,0) {\scriptsize $\cdots$};
\node at (2.4,0)[Tbox, minimum height=.5cm, minimum width=.5cm]{\scriptsize $b_{j_n}$};
\node at (.1,.45) {\scriptsize $\$$};
\node at (.4,.33) {\scriptsize $\$$};
\node at (1.2,.33) {\scriptsize $\$$};
\node at (2.4,.33) {\scriptsize $\$$};
\draw[rounded corners, very thick] (0,-.5) rectangle (2.8,.6);
\end{tikzpicture}.$$
Since elements of the same form as $x$ for different indices $i_1,...,i_n$ form an orthogonal basis of $P_{n,+}^{Spin}$, we obtain $\pi_{x,y} \in C^*(M,e_N,JMJ)$ for any $x,y \in P_{n,+}^{Spin}$.
\end{proof}

\begin{lem}\label{norm}
Fix $n,l \in \mathbb N$ and let $x,y \in P_{n,+}^{Spin}$. Define the linear operator $\rho_{x,y,l} \colon \bigcup_k P^{Spin}_{k+2l,+} \to \bigcup_k P^{Spin}_{k+2l,+}$
$$\rho_{x,y,l} \xi = 
\begin{tikzpicture}[align1]
\clip[rounded corners] (.9,-1.8) rectangle (3.5,1.8);

\draw (0,1.4) rectangle (3.75,-1.4);
\draw[rounded corners] (1.2,1) rectangle (2.2,-1);

\draw[line width=1mm] (1.8,.65)--(1.8,-.65);
\draw[line width=1mm] (0,0)--(4,0);
\draw[line width=1mm] (0,1.2)--(3.75,1.2);
\draw[line width=1mm] (0,-1.2)--(3.75,-1.2);

\draw (0,1.4) rectangle (3.75,-1.4);

\path[shadedw] (3,2) rectangle (4.5,-2);

\draw[->] (1.2,.5)--(1.2,.49);
\draw[->] (1.65,-.3)--(1.65,-.15);

\node at (3.35,0) {\scriptsize $\$$};
\node at (1,1.65) {\scriptsize $\$$};
\node at (1.45,.65) {\scriptsize $\$$};
\node at (1.45,-.65) {\scriptsize $\$$};
\node at (3,0)[Tbox, minimum width=.5cm, minimum height=3.2cm]{\scriptsize $\xi$};
\node at (1.8,.65)[Tbox, minimum height=.5cm, minimum width=.5cm]{\scriptsize $x$};
\node at (1.8,-.65)[Tbox, minimum height=.5cm, minimum width=.5cm]{\scriptsize $y^*$};
\node at (1.8,.225)[Tcirc, inner sep=0cm]{\tiny $2n$};
\node at (2.5,0)[Tcirc, inner sep=0cm]{\scriptsize $2k$};
\node at (2.5,1.2)[Tcirc, inner sep=0cm]{\scriptsize $2l$};
\node at (2.5,-1.2)[Tcirc, inner sep=0cm]{\scriptsize $2l$};
\draw[rounded corners, very thick] (.9,-1.8) rectangle (3.5,1.8);
\end{tikzpicture} \quad \text{ for } \xi \in P^{Spin}_{k+2l+1,+}.$$
Then $\rho_{x,y,l}$ extends uniquely to a bounded operator on $L^2(M)$ with
$$\norm{\rho_{x,y,l}}_{B(L^2(M))} \leq \sqrt{Q} \norm{x}_{Spin} \norm{y}_{Spin}.$$
\end{lem}

\begin{proof}
We will verify this inequality by obtaining an upper bound of $\left| \inn{\rho_{x,y,l}\xi}{\eta}_{tr_M} \right|$ for $\xi$ and $\eta$ in a dense subset of $L^2(M)$. Let $\xi, \eta \in P^{Spin}_{k+2l+1,+}$, then due to the Cauchy-Schwarz inequality and unitarity of the first diagram below

$$\left| \inn{\rho_{x,y,l}\xi}{\eta}_{Spin} \right| = 
\left|
\,\,\,
\begin{tikzpicture}[align1]
\draw (0,-1.8)--(-.2,0)--(0,1.8);
\end{tikzpicture}
\,\,\,\,
\begin{tikzpicture}[align1]
\clip[rounded corners] (.9,-1.8) rectangle (2.8,1.8);

\draw (0,1.4) rectangle (3.75,-1.4);
\draw (2.8,.3) to [out=180,in=90] (2.4,.05)--(2.4,-.8) to [out=-90,in=0] (2.2,-1)--(.9,-1)--(.9,-.3) to [out=0,in=-90] (1.2,-.05)--(1.2,.8) to [out=90,in=180] (1.4,1)--(2.8,1)--(2.8,.3)--cycle;

\draw[line width=1mm] (3,.65) to [out=180,in=90] (1.8,.05)--(1.8,-.05) to [out=-90,in=0] (.9,-.65);
\draw[line width=1mm] (0,0)--(4,0);
\draw[line width=1mm] (0,1.2)--(3.75,1.2);
\draw[line width=1mm] (0,-1.2)--(3.75,-1.2);

\draw (2.8,.3) to [out=180,in=90] (2.4,.05)--(2.4,-.8) to [out=-90,in=0] (2.2,-1)--(.9,-1)--(.9,-.3) to [out=0,in=-90] (1.2,-.05)--(1.2,.8) to [out=90,in=180] (1.4,1)--(2.8,1)--(2.8,.3)--cycle;

\draw[->] (1.2,-.45) to [out=0,in=220] (1.5,-.3);
\draw[->] (1.2,.55)--(1.2,.54);
\draw[->] (2.4,-.55)--(2.4,-.54);

\node at (1,1.65) {\scriptsize $\$$};
\node at (2,.45)[Tcirc, inner sep=0cm]{\scriptsize $2n$};
\node at (2.15,0)[Tcirc, inner sep=0cm]{\scriptsize $2k$};
\node at (2.5,1.2)[Tcirc, inner sep=0cm]{\scriptsize $2l$};
\node at (2.5,-1.2)[Tcirc, inner sep=0cm]{\scriptsize $2l$};
\draw[rounded corners, very thick] (.9,-1.8) rectangle (2.8,1.8);
\end{tikzpicture}
\,\,\,\,
\begin{tikzpicture}[align1]
\clip[rounded corners] (.9,-1.8) rectangle (3.5,1.8);

\draw (0,1.4) rectangle (3.75,-1.4);
\draw[rounded corners] (0,1) rectangle (2.4,.3);

\draw[line width=1mm] (1.8,.65)--(0,.65);
\draw[line width=1mm] (0,0)--(4,0);
\draw[line width=1mm] (0,1.2)--(3.75,1.2);
\draw[line width=1mm] (0,-1.2)--(3.75,-1.2);

\draw[rounded corners] (0,1) rectangle (2.4,.3);

\path[shadedw] (3,2) rectangle (4.5,-2);

\node at (3.35,0) {\scriptsize $\$$};
\node at (1,1.65) {\scriptsize $\$$};
\node at (2.15,.65) {\scriptsize $\$$};
\node at (3,0)[Tbox, minimum width=.5cm, minimum height=3.2cm]{\scriptsize $\xi$};
\node at (1.8,.65)[Tbox, minimum height=.5cm, minimum width=.5cm]{\scriptsize $x$};
\node at (1.25,.65)[Tcirc, inner sep=0cm]{\scriptsize $2n$};
\node at (2.5,0)[Tcirc, inner sep=0cm]{\scriptsize $2k$};
\node at (2.5,1.2)[Tcirc, inner sep=0cm]{\scriptsize $2l$};
\node at (2.5,-1.2)[Tcirc, inner sep=0cm]{\scriptsize $2l$};
\draw[rounded corners, very thick] (.9,-1.8) rectangle (3.5,1.8);
\end{tikzpicture}
\,\,\,\,
\begin{tikzpicture}[align1]
\draw (0,-1.8)--(0,1.8);
\end{tikzpicture}
\,\,\,\,
\begin{tikzpicture}[align1]
\clip[rounded corners] (.9,-1.8) rectangle (3.5,1.8);

\draw (0,1.4) rectangle (3.75,-1.4);
\draw[rounded corners] (0,-1) rectangle (2.4,-.3);

\draw[line width=1mm] (1.8,-.65)--(0,-.65);
\draw[line width=1mm] (0,0)--(4,0);
\draw[line width=1mm] (0,1.2)--(3.75,1.2);
\draw[line width=1mm] (0,-1.2)--(3.75,-1.2);

\draw[rounded corners] (0,-1) rectangle (2.4,-.3);

\path[shadedw] (3,2) rectangle (4.5,-2);

\node at (3.35,0) {\scriptsize $\$$};
\node at (1,1.65) {\scriptsize $\$$};
\node at (2.15,-.65) {\scriptsize $\$$};
\node at (3,0)[Tbox, minimum width=.5cm, minimum height=3.2cm]{\scriptsize $\eta$};
\node at (1.8,-.65)[Tbox, minimum height=.5cm, minimum width=.5cm]{\scriptsize $y$};
\node at (1.25,-.65)[Tcirc, inner sep=0cm]{\scriptsize $2n$};
\node at (2.5,0)[Tcirc, inner sep=0cm]{\scriptsize $2k$};
\node at (2.5,1.2)[Tcirc, inner sep=0cm]{\scriptsize $2l$};
\node at (2.5,-1.2)[Tcirc, inner sep=0cm]{\scriptsize $2l$};
\draw[rounded corners, very thick] (.9,-1.8) rectangle (3.5,1.8);
\end{tikzpicture} 
\,\,\,\,
\begin{tikzpicture}[align1]
\draw (0,-1.8)--(.2,0)--(0,1.8);
\end{tikzpicture}
\,\,\,\,
\right|$$
$$\leq
\,\,\,\,
\begin{tikzpicture}[align1]
\begin{scope}[shift=(-90:.5)]
\clip[rounded corners] (0,-.8) rectangle (3.5,1.8);

\draw (.5,1.4) rectangle (3,-.4);
\draw[rounded corners] (.85,1) rectangle (2.65,.3);

\draw[line width=1mm] (1.35,.65)--(2.15,.65);
\draw[line width=1mm] (.5,.1)--(3,.1);
\draw[line width=1mm] (.5,1.2)--(3,1.2);
\draw[line width=1mm] (.5,-.2)--(3,-.2);

\draw (.5,1.4) rectangle (3,-.4);
\draw[rounded corners] (.85,1) rectangle (2.65,.3);

\node at (3.35,.5) {\scriptsize $\$$};
\node at (.15,.5) {\scriptsize $\$$};
\node at (.1,1.65) {\scriptsize $\$$};
\node at (.95,.65) {\scriptsize $\$$};
\node at (2.55,.65) {\scriptsize $\$$};
\node at (3,.5)[Tbox, minimum width=.5cm, minimum height=2.2cm]{\scriptsize $\xi$};
\node at (.5,.5)[Tbox, minimum width=.5cm, minimum height=2.2cm]{\scriptsize $\xi^*$};
\node at (2.2,.65)[Tbox, minimum height=.5cm, minimum width=.5cm]{\scriptsize $x$};
\node at (1.3,.65)[Tbox, minimum height=.5cm, minimum width=.5cm]{\scriptsize $x^*$};
\node at (1.75,.65)[Tcirc, inner sep=0cm]{\tiny $2n$};
\node at (1.55,.1)[Tcirc, inner sep=0cm]{\scriptsize $2k$};
\node at (1.75,1.2)[Tcirc, inner sep=0cm]{\scriptsize $2l$};
\node at (1.75,-.2)[Tcirc, inner sep=0cm]{\scriptsize $2l$};
\draw[rounded corners, very thick] (0,-.8) rectangle (3.5,1.8);
\end{scope}
\end{tikzpicture}^{1/2}
\cdot
\begin{tikzpicture}[align1]
\begin{scope}[shift=(90:.5)]
\clip[rounded corners] (0,-1.8) rectangle (3.5,.8);

\draw (.5,.4) rectangle (3,-1.4);
\draw[rounded corners] (.85,-1) rectangle (2.65,-.3);

\draw[line width=1mm] (1.35,-.65)--(2.15,-.65);
\draw[line width=1mm] (.5,-.1)--(3,-.1);
\draw[line width=1mm] (.5,.2)--(3,.2);
\draw[line width=1mm] (.5,-1.2)--(3,-1.2);

\draw[rounded corners] (.85,-1) rectangle (2.65,-.3);

\node at (3.35,-.5) {\scriptsize $\$$};
\node at (.15,-.5) {\scriptsize $\$$};
\node at (.1,.65) {\scriptsize $\$$};
\node at (.95,-.65) {\scriptsize $\$$};
\node at (2.55,-.65) {\scriptsize $\$$};
\node at (3,-.5)[Tbox, minimum width=.5cm, minimum height=2.2cm]{\scriptsize $\eta$};
\node at (.5,-.5)[Tbox, minimum width=.5cm, minimum height=2.2cm]{\scriptsize $\eta^*$};
\node at (2.2,-.65)[Tbox, minimum height=.5cm, minimum width=.5cm]{\scriptsize $y$};
\node at (1.3,-.65)[Tbox, minimum height=.5cm, minimum width=.5cm]{\scriptsize $y^*$};
\node at (1.75,-.65)[Tcirc, inner sep=0cm]{\tiny $2n$};
\node at (1.55,-.1)[Tcirc, inner sep=0cm]{\scriptsize $2k$};
\node at (1.75,.2)[Tcirc, inner sep=0cm]{\scriptsize $2l$};
\node at (1.75,-1.2)[Tcirc, inner sep=0cm]{\scriptsize $2l$};
\draw[rounded corners, very thick] (0,-1.8) rectangle (3.5,.8);
\end{scope}
\end{tikzpicture}^{1/2}
\leq 
\sqrt{Q} \norm{x}_{Spin} \norm{y}_{Spin} \norm{\xi}_{Spin} \norm{\eta}_{Spin}.$$
Since the tracial inner product is a normalization of $\inn{\cdot}{\cdot}_{Spin}$, we have
$$\left| \inn{\rho_{x,y,l}\xi}{\eta}_{tr_M} \right| \leq \norm{x}_{Spin} \norm{y}_{Spin} \norm{\xi}_{tr_M}\norm{\eta}_{tr_M}$$
for all $\xi,\eta$ in a dense subset of $L^2(M)$. Therefore $\rho_{x,y,l}$ extends uniquely to a bounded operator with $\norm{\rho_{x,y,l}}_{B(L^2(M))} \leq \sqrt{Q} \norm{x}_{Spin} \norm{y}_{Spin}$.
\end{proof}

\begin{lem}
Let $\{b_i\}_{i=1}^{Q^n}$ be an orthonormal basis of $P^{Spin}_{n,+}$ with respect to $\inn{\cdot}{\cdot}_{Spin}$, then $\displaystyle\sum_{i=1}^{Q^n} b_i b_i^* = \sqrt{Q}^n \cdot id_{P^{Spin}_{n,+}}.$
\end{lem}

This lemma follows from the cable cutting lemma \ref{cabling} and the loop parameters for $P^{Spin}$.

\begin{prop}\label{tau approximation}
Let $\tau$ denote the unique continuous trace on $C^*(M,e_N,JMJ)$ constructed in \cite{Pop99}. Then for $x,y \in P^{Spin}_{n,+}$ and any $l \in \mathbb N$
$$\abs{\tau(\pi_{x,y})} \leq \dfrac{1}{\sqrt{Q}} \norm{\theta_{u^*}^l(x)}_{Spin} \norm{\theta_{u^*}^l(y)}_{Spin}.$$
\end{prop}

\begin{proof}
Let $\{b_i\}_{i=1}^{Q^{2l+1}}$ be an orthonormal basis of $P^{Spin}_{2l+1,+}$. Since $\tau$ is a trace\\
$\tau(\pi_{x,y})=\sum_{i,j=1}^{Q^{2l+1}} \dfrac{1}{\sqrt{Q}^{4l+2}} \tau(b_iJb_j^*J\pi_{x,y}b_i^*Jb_jJ)$. Let $\xi \in \bigcup_k P^{Spin}_{2l+k+1,+}$ then
$$\sum_{i,j=1}^{Q^{2l+1}} \dfrac{1}{\sqrt{Q}^{4l+2}} b_iJb_j^*J\pi_{x,y}b_i^*Jb_jJ \xi=\sum_{i,j=1}^{Q^{2l+1}} \dfrac{1}{\sqrt{Q}^{4l+2}}
\begin{tikzpicture}[align1]
\clip[rounded corners] (-1.5,-1.4) rectangle (2.8,1.4);
\draw[rounded corners] (.5,.8) rectangle (5,-.8);
\draw[rounded corners] (-5,.8) rectangle (-.5,-.8);
\draw[line width=1mm] (0,.9)--(0,-.9);
\draw[line width=1mm] (-5,0)--(5,0);
\draw[line width=1mm] (-5,.5)--(5,.5);
\draw[line width=1mm] (-5,-.5)--(5,-.5);
\draw[->] (-.5,-.25)--(-.5,-.251);
\draw[->] (.5,-.251)--(.5,-.25);
\draw[->] (-.17,-.35)--(-.17,-.15);
\path[shadedw] (2.3,2) rectangle (3,-2);
\node at (-1.4,1.25) {\scriptsize $\$$};
\node at (2.65,0) {\scriptsize $\$$};
\node at (0,1.25) {\scriptsize $\$$};
\node at (0,-1.25) {\scriptsize $\$$};
\node at (1.1,1) {\scriptsize $\$$};
\node at (1.1,-1) {\scriptsize $\$$};
\node at (-1.1,1) {\scriptsize $\$$};
\node at (-1.1,-1) {\scriptsize $\$$};
\node at (2.3,0)[Tbox, minimum width=.5cm, minimum height=2cm]{\scriptsize $\xi$};
\node at (0,.9)[Tbox, minimum height=.5cm, minimum width=.5cm]{\scriptsize $x$};
\node at (0,-.9)[Tbox, minimum height=.5cm, minimum width=.5cm]{\scriptsize $y^*$};
\node at (1.1,.65)[Tbox, minimum height=.5cm, minimum width=0cm]{\scriptsize $b_i^*$};
\node at (-1.1,.65)[Tbox, minimum height=.5cm, minimum width=0cm]{\scriptsize $b_i$};
\node at (-1.1,-.65)[Tbox, minimum height=.5cm, minimum width=0cm]{\scriptsize $b_j$};
\node at (1.1,-.65)[Tbox, minimum height=.5cm, minimum width=0cm]{\scriptsize $b_j^*$};
\node at (0,.25)[Tcirc, inner sep=0cm]{\scriptsize $2n$};
\node at (1.7,0)[Tcirc, inner sep=0cm]{\scriptsize $2k$};
\node at (1.7,.5)[Tcirc, inner sep=0cm]{\scriptsize $2l$};
\node at (1.7,-.5)[Tcirc, inner sep=0cm]{\scriptsize $2l$};
\draw[rounded corners, very thick] (-1.5,-1.4) rectangle (2.8,1.4);
\end{tikzpicture}$$
$$=\dfrac{1}{\sqrt{Q}^{4l+2}} \begin{tikzpicture}[align1]
\clip[rounded corners] (-1.2,-2.1) rectangle (1.5,2.1);

\draw (-5,1.8) rectangle (5,-1.8);
\draw[rounded corners] (-.9,-1.4) rectangle (.6,1.4);

\draw[line width=1mm] (0,.9)--(0,-.9);
\draw[line width=1mm] (-5,0)--(5,0);
\draw[line width=1mm] (-5,1.6)--(5,1.6);
\draw[line width=1mm] (-5,-1.6)--(5,-1.6);

\draw[line width=1mm, rounded corners] (-.6,.5) rectangle (.4,1.2);
\draw[line width=1mm, rounded corners] (-.6,-.5) rectangle (.4,-1.2);

\draw[->] (-.9,-.25)--(-.9,-.251);
\draw[->] (-.15,-.35)--(-.15,-.15);

\path[shadedw] (1,2.5) rectangle (3,-2.5);

\node at (-1.1,1.95) {\scriptsize $\$$};
\node at (1.35,0) {\scriptsize $\$$};
\node at (-.35,.85) {\scriptsize $\$$};
\node at (-.35,-.85) {\scriptsize $\$$};

\node at (1,0)[Tbox, minimum width=.5cm, minimum height=4.1cm]{\scriptsize $\xi$};
\node at (0,.85)[Tbox, minimum height=.5cm, minimum width=.5cm]{\scriptsize $x$};
\node at (0,-.85)[Tbox, minimum height=.5cm, minimum width=.5cm]{\scriptsize $y^*$};

\node at (0,.25)[Tcirc, inner sep=0cm]{\scriptsize $2n$};
\node at (-.4,0)[Tcirc, inner sep=0cm]{\scriptsize $2k$};
\node at (-.6,.85)[Tcirc, inner sep=0cm]{\scriptsize $2l$};
\node at (-.6,-.85)[Tcirc, inner sep=0cm]{\scriptsize $2l$};
\node at (-.4,1.6)[Tcirc, inner sep=0cm]{\scriptsize $2l$};
\node at (-.4,-1.6)[Tcirc, inner sep=0cm]{\scriptsize $2l$};

\draw[rounded corners, very thick] (-1.2,-2.1) rectangle (1.5,2.1);
\end{tikzpicture}=\dfrac{1}{Q} \rho_{\theta_{u^*}^l(x),\theta_{u^*}^l(y),l}.$$
Since $\tau$ is norm continuous and using Lemma \ref{norm}, the proposition follows.
\end{proof}

\begin{prop}
$\tau(\Theta_u^n)=\dfrac{dim(N' \cap M_{n-1})}{Q^{n+1}}$ for $n \geq 1$.
\end{prop}

\begin{proof}
Let $\{b_i\}_{i=1}^Q$ be an orthonormal basis of $P_{1,+}^{Spin}$ then by Proposition \ref{cond exp}
$$\Theta_u \xi = \dfrac{1}{\sqrt{Q}^3}
\begin{tikzpicture}[align1]
\clip (0,0) circle (1.8cm);
\draw[line width=1mm] (-2,0)--(0,0);
\draw (-.3,0) circle (.3cm);
\draw (-1.8,0) circle (.3cm);
\draw (0,0) circle (.8cm);
\draw(0,0) circle (1.3cm);
\draw[<-] (.8,.01)--(.8,0);
\draw[<-] (1.3,0)--(1.3,.01);
\draw[->] (-1.587,.212)--(-1.586,.211);
\draw[->] (-.512,.212)--(-.511,.213);
\node at (0,0)[Tcirc, inner sep=.2cm]{\scriptsize $\xi$};
\node at (-1.05,0)[Tcirc, inner sep=0cm]{\scriptsize $2k$};
\node at (.455,0) {\scriptsize $\$$};
\node at (1.7,0) {\scriptsize $\$$};
\draw[very thick] (0,0) circle (1.8cm);
\end{tikzpicture}=\dfrac{1}{\sqrt{Q}^3}\sum_{i=1}^Q \pi_{b_i,b_i}\xi \quad \text{ for } \xi \in P^{Spin}_{k+1,+}$$
and so $\Theta_u  \in C^*(M,e_N,JMJ)$ and $\tau(\Theta_u^n)$ is well-defined. Similarly, if $\{b_i\}_{i=1}^{Q^n}$ is an orthonormal basis of $P_{n,+}^{Spin}$ then 
$$\Theta_u^n \xi = \dfrac{1}{Q^{n+1/2}}
\begin{tikzpicture}[align1]
\clip[rounded corners] (-1.8,-1.1) rectangle (1.1,1.1);
\draw[line width=1mm] (-2,0)--(0,0);
\draw (-.3,0) circle (.3cm);
\draw (-1.8,0) circle (.3cm);
\draw[line width=1mm] (0,0) circle (.9cm);
\draw[->] (150:1.1cm) arc(150:130:1cm) -- (130:1.1cm);
\draw[->] (-1.587,.212)--(-1.586,.211);
\draw[->] (-.512,.212)--(-.511,.213);
\node at (0,0)[Tcirc, inner sep=.2cm]{\scriptsize $\xi$};
\node at (-1.25,0)[Tcirc, inner sep=0cm]{\scriptsize $2k$};
\node at (225:.9cm)[Tcirc, inner sep=0cm]{\scriptsize $2n$};
\node at (.455,0) {\scriptsize $\$$};
\node at (-1.7,.95) {\scriptsize $\$$};
\draw[very thick, rounded corners] (-1.8,-1.1) rectangle (1.1,1.1);
\end{tikzpicture}=\dfrac{1}{Q^{n+1/2}}\sum_{i=1}^{Q^n} \pi_{b_i,b_i}\xi \quad  \text{ for } \xi \in P^{Spin}_{k+1,+}.$$
Let $d_n=dim(P_{n,+}^{u^*})$ and choose an orthonormal basis, $\{f_i\}_{i=1}^{d_n}$ of $P_{n,+}^{u^*}$ and an orthonormal basis, $\{b_j\}_{j=1}^{Q^n-d_n}$ of $\left(P_{n,+}^{u^*}\right)^\perp \cap P_{n,+}^{Spin}$. Since the $f_i$'s are flat, $\pi_{f_i,f_i} = \sqrt{Q} \inn{f_i}{f_i}_{Spin} e_N=\sqrt{Q}e_N$ and so
$$\tau(\Theta_u^n)=\dfrac{1}{Q^{n+1/2}}\sum_{i=1}^{d_n} \tau(\pi_{f_i,f_i})+\dfrac{1}{Q^{n+1/2}}\sum_{j=1}^{Q^n -d_n} \tau(\pi_{b_j,b_j}) = \dfrac{dim(N' \cap M_{n-1})}{Q^{n+1}} + \dfrac{1}{Q^{n+1/2}}\sum_{j=1}^{Q^n -d_n} \tau(\pi_{b_j,b_j}).$$
It suffices to show that $\tau(\pi_{b_j,b_j})=0$. By Proposition \ref{tau approximation} $\abs{\tau(\pi_{b_j,b_j})} \leq \norm{\theta_{u^*}^l(b_j)}_{Spin}^2$ for any $l \in \mathbb N$. Since the $b_j$'s are orthogonal to the eigenspace of $\theta_{u^*}$ corresponding to the eigenvalue $\lambda=1$ and $\theta_{u^*}|_{P_{n,+}^{Spin}}$ is a positive, diagonalizable operator with norm less than or equal to one, $\norm{\theta_{u^*}^l(b_j)}_{Spin} \leq (1-\varepsilon)^l$ for some $0< \varepsilon < 1$. Therefore $\tau(\pi_{b_j,b_j})=0$.
\end{proof}

\begin{thm}
$\sigma(\Gamma\Gamma^*) \subset \overline{\bigcup_n \sigma(Q\theta_u|_{P_{n,+}^{Spin}})}$ with equality iff $N \subset M$ is amenable where $\Gamma$ is the principal graph of $N \subset M$.
\end{thm}

\begin{proof}
Observe that $\Theta_u=e_N\Theta_ue_N$ and so $\Theta_u$ belongs to the corner algebra $e_NC^*(M,e_N,JMJ)e_N$ which is faithfully represented on $e_NL^2(M)$. Using the unitary, $\psi_u \colon L^2(N) \to e_NL^2(M)$, defined in Proposition \ref{psi}, we may represent $e_NC^*(M,e_N,JMJ)e_N$ on $L^2(N)$ by $\lambda \colon e_NC^*(M,e_N,JMJ)e_N \to B(L^2(N))$, $\lambda(x)\xi = \psi_u^*x\psi_u\xi$. Set $\mathcal S =\lambda(e_NC^*(M,e_N,JMJ)e_N) \subset B(L^2(N))$ and define a tracial state $\tilde{\tau} \colon \mathcal S \to \mathbb C$, $\tilde{\tau}(x) = Q\tau(\psi_u x \psi_u^*)$. Since $\psi_u\psi_u^*=e_N$, $\Theta_u \psi_u = \psi_u \theta_u$, and $\tau(e_N)=tr_{M_1}(e_N) = \frac{1}{Q}$ then $\theta_u \in \mathcal S$ and $\tilde{\tau}$ is a normalized trace with $\tilde{\tau}(\theta_u^n)=\dfrac{dim(N' \cap M_{n-1})}{Q^n}$ for all $n \geq 0$.

Let $\Gamma$ be the principal graph of $N \subset M$. $\Gamma\Gamma^*$ defines a bounded linear operator in $B(L^2(V(\Gamma_{even})))$ where $L^2(V(\Gamma_{even}))$ has the even vertices as an orthonormal basis. $C^*(1,\Gamma\Gamma^*)$ comes with a state $\phi(x)=\inn{x\delta_*}{\delta_*}$ where $\delta_*$ is the indicator function on the distinguished vertex of $\Gamma$. Frobenius reciprocity in the fusion algebra of $N\subset M$ implies that $\phi$ is faithful. By the Riesz-Markov-Kakutani representation theorem, $\phi$ (resp. $\tilde{\tau}$) induce unique positive Radon measures, $d\phi$ (resp. $d\tilde{\tau}$) on the spectrum of $\Gamma\Gamma^*$ (resp. $Q\theta_u$). Since these spectra are compact subsets of $[0,Q]$, we may consider $d\phi$ (resp. $d\tilde{\tau}$) as positive Radon measures on $[0,Q]$ by $d\phi(E)=d\phi(E \cap \sigma(\Gamma\Gamma^*))$ (resp. for $d\tilde{\tau}$). Since $\phi((\Gamma\Gamma^*)^n)=dim(N' \cap M_{n-1})$, the moments of $d\phi$ and $d\tilde{\tau}$ are equal,
$$\int_0^Q \lambda^n d\phi(\lambda) = \int_0^Q \lambda^n d\tilde{\tau}(\lambda) \quad \text{for all } n \geq 0$$
and so by the Stone-Weierstrass theorem these measures define the same continuous linear functionals on $C([0,Q])$. Then by faithfulness of $\phi$
$$\sigma(\Gamma\Gamma^*)= supp(d\phi) = supp(d\tilde{\tau}) \subset \sigma(Q\theta_u).$$
If $N \subset M$ is amenable then, due to Popa, $\tilde{\tau}$ is also faithful yielding equality of the spectra. If $N \subset M$ is not amenable then $\norm{\Gamma}^2 < \norm{Q\theta_u}=Q$ and so their spectra cannot be equal. $\sigma(Q\theta_u) = \overline{\bigcup_n \sigma(Q\theta_u|_{P_{n,+}^{Spin}})}$ remains to be shown.

$\overline{\bigcup_n \sigma(Q\theta_u|_{P_{n,+}^{Spin}})} \subset \sigma(Q\theta_u)$ is trivially true. Now let $r \notin \sigma(Q\theta_u)$ and observe that $\norm{\dfrac{1}{r-Q\theta_u}} \leq \dfrac{1}{dist(r,\sigma(Q\theta_u))}$ by continuous functional calculus. Since $r-Q\theta_u$ maps $P_{n,+}^{Spin}$ bijectively onto $P_{n,+}^{Spin}$, then $\norm{\dfrac{1}{(r-Q\theta_u)|_{P_{n,+}^{Spin}}}} \leq \norm{\left. \dfrac{1}{r-Q\theta_u} \right|_{P_{n,+}^{Spin}}}$. Since $Q\theta_u|_{P_{n,+}^{Spin}}$ is diagonalizable, $r \notin \overline{\bigcup_n \sigma(Q\theta_u|_{P_{n,+}^{Spin}})}$, and so $\sigma(\Gamma \Gamma^*) = \overline{\bigcup_n \sigma(Q\theta_u|_{P_{n,+}^{Spin}})}$.
\end{proof}

This provides us with two computational tools. First, if we already know a spin model subfactor is amenable then we can compute elements in the spectrum of its principal graph. Second, since the spectrum of a finite graph is contained in the algebraic integers, we may prove that a spin model subfactor is infinite depth by finding non-algebraic integers in the spectrum of $Q\theta_u$.

\section{Applications}
First, we consider continuous families of complex Hadamard matrices, $u_t$. Such a family yields a continuous family of angle operators $\theta_{u_t}|_{P_{n,+}^{Spin}}$ for each $n \geq 0$. This will imply that infinite depth subfactors are a generic feature of continuous families of complex Hadamards. For a von Neumann algebra, $A$, let $(A)_1$ denote the unit ball and let
$$D(A,B)=\sup \set{ \inf_{x \in (B)_1} \norm{a-x}, \inf_{x \in (A)_1} \norm{b-x}}{a \in (A)_1 \text{ and } b \in (B)_1}$$
denote the Hausdorff metric between two von Neumann algebras, $A,B \subset B(H)$.
Then in \cite{Phi74}, Phillips shows that if $D(A,B) <\varepsilon (\leq \frac{1}{12})$ then $Z(A) \cong Z(B)$. The isomorphism $\varphi \colon Z(A) \to Z(B)$ is given by $\varphi(p)=q$ where $q$ is the unique central projection in $B$ such that $\norm{p-q}<\frac{1}{3}$. Furthermore, if $D(A,B) < \frac{1}{25618}$ and $A$ is a type $I$ von Neumann algebra then $A$ and $B$ are unitarily equivalent. See also \cite{Chr79} for similar results for type $II_1$ von Neumann algebras using the trace norm instead of the operator norm.

\begin{lem}
Let $A_0 \subset A_1 \subset A_2 \subset B(H)$ and $B_0 \subset B_1 \subset B_2 \subset B(H)$ be finite dimensional $C^*$-algebras with unital inclusions. Suppose there exists a projection $e \in A_2 \cap B_2$ implementing the unique conditional expectations $E^{A_1}_{A_0}$ and $E^{B_1}_{B_0}$ with respect to their Markov traces. Further assume that $A_2=\{A_1,e\}''$, $B_2=\{B_1,e\}''$ and that $D(A_i,B_i) < \frac{1}{265180}$ for $i=0,1,2$. Then the bijection between minimal central projections above induces an isomorphism between $\Gamma^{A_i \subset A_{i+1}}$ and $\Gamma^{B_i \subset B_{i+1}}$ and commutes with the map $p \mapsto pe$ for $p \in Z(A_0)$, sending $p$ to a minimal central projection in $A_2$.
\end{lem}

\begin{proof}
Observe that the inclusion matrix for $A_0 \subset A_1$, is given by $\Gamma^{A_0\subset A_1} = (\gamma_{p,q})_{p,q \text{ min. central proj.}}$ where $\gamma_{p,q} = 0$ if $pq=0$ and $\gamma_{p,q} = \sqrt{\frac{dim(pqA_1pq)}{dim(pqA_0pq)}}$ otherwise. Thus, if $D(A_0,B_0)$ and $D(A_1,B_1)$ are sufficiently small then the centers of $A_i$ and $B_i$ can be identified. Furthermore,\\ $D(pqA_ipq,\varphi(pq)B_i\varphi(pq)) < \frac{1}{25618}$ for any $p \in Z(A_0)$ and $q\in Z(A_1)$. This implies that the inclusion matrices for $A_0 \subset A_1$ and $B_0 \subset B_1$ are isomorphic. The same argument applies to the other inclusions.

The last claim follows from the observation $\norm{pe - \varphi(p)e} \leq \norm{p-\varphi(p)} < \frac{1}{3}$.
\end{proof}

\begin{prop}
Let $H \colon \mathbb R \to M_Q(\mathbb C)$, $t \mapsto H_t$, be a continuous family of complex Hadamard matrices. Then one of the following is true:
\begin{enumerate}
\item The corresponding principal graphs are equal for all $t \in \mathbb R$.
\item There are uncountably many $t\in \mathbb R$ such that the corresponding subfactors are infinite depth.
\end{enumerate}
\end{prop}

\begin{proof}
Given $t \mapsto H_t$, let $t \mapsto u_t$ be the corresponding biunitaries. Then for all fixed $n \geq 0$, $t \mapsto Q\theta_{u_t}|_{P^{Spin}_{n,+}}$ is a continuous map to positive finite dimensional matrices. Since the spectra of positive matrices vary continuously in the Hausdorff metric, if $t \mapsto \sigma( Q\theta_{u_t}|_{P^{Spin}_{n,+}})$ is not constant then uncountably many $t$ yield infinite depth subfactors.

Now suppose that  $t \mapsto \sigma( Q\theta_{u_t}|_{P^{Spin}_{n,+}})$ is constant for all $n \geq 0$. Since the spectrum is constant and $\sigma(\theta_{u_t}|_{P^{Spin}_{n,+}}) = \sigma(\theta_{u_t^*}|_{P^{Spin}_{n,+}})$, the $1$-eigenspaces, $P^{u_t^*}_{n,+} \subset P^{Spin}_{n,+}$, vary continuously in the metric $D(A,B)$ defined above. Letting $N_t \subset M_t$ denote the spin model subfactor from $u_t$, the previous lemma implies that
$$S_{s,n} = \set{t \in \mathbb R}{\Gamma^{N_t \subset M_t} \text{ is isomorphic up to depth $n$ to } \Gamma^{N_s \subset M_s}}$$
is an open subset of $\mathbb R$ for all $n$. By connectedness $S_{s,n}=\mathbb R$ for all $n$ and so $(1)$ is true.
\end{proof}

\begin{exmp}
In \cite{Pet97} Petrescu constructs a continuous family of inequivalent $7 \times 7$ complex Hadamard matrices given by
$$H=\left[\begin{array}{ccccccc}
	\lambda \omega & \lambda \omega^4 & \omega^5 & \omega^3 & \omega^3 & \omega & 1\\
	\lambda \omega^4 & \lambda \omega & \omega^3 & \omega^5 & \omega^3 & \omega & 1\\
	\omega^5 & \omega^3 & \overline{\lambda} \omega & \overline{\lambda} \omega^4 & \omega & \omega^3 & 1\\
    \omega^3 & \omega^5 & \overline{\lambda} \omega^4 & \overline{\lambda} \omega & \omega & \omega^3 & 1\\
	\omega^3 & \omega^3 & \omega & \omega & \omega^4 & \omega^5 & 1\\
	\omega & \omega & \omega^3 & \omega^3 & \omega^5 & \omega^4 & 1\\
	1 & 1 & 1 & 1 & 1 & 1 & 1\\
\end{array}\right]$$
where $\omega = e^{i\pi/3}$ and $\lambda \in \mathbb T$. Letting $u$ be the corresponding biunitary in $P^{Spin}$, $7\theta_u|_{P^{Spin}_{2,+}}$ has an eigenvector given by
$$\xi=\left[\begin{array}{ccccccc}
	0 & 0 & 1 & -1 & \frac{1}{\sqrt{3}}Im(\lambda \overline{\omega}) & \frac{-1}{\sqrt{3}}Im(\lambda) & \frac{1}{\sqrt{3}}Im(\lambda \omega)\\
	0 & 0 & -1 & 1 & \frac{1}{\sqrt{3}}Im(\lambda \overline{\omega}) & \frac{-1}{\sqrt{3}}Im(\lambda) & \frac{1}{\sqrt{3}}Im(\lambda \omega)\\
	1 & -1 & 0 & 0 & \frac{1}{\sqrt{3}}Im(\lambda) & \frac{-1}{\sqrt{3}}Im(\lambda \omega) & \frac{-1}{\sqrt{3}}Im(\lambda \overline{\omega})\\
	-1 & 1 & 0 & 0 & \frac{1}{\sqrt{3}}Im(\lambda) & \frac{-1}{\sqrt{3}}Im(\lambda \omega) & \frac{-1}{\sqrt{3}}Im(\lambda \overline{\omega})\\
	\frac{1}{\sqrt{3}}Im(\lambda \overline{\omega}) & \frac{1}{\sqrt{3}}Im(\lambda \overline{\omega}) & \frac{1}{\sqrt{3}}Im(\lambda) & \frac{1}{\sqrt{3}}Im(\lambda) & 0 & 2Re(\lambda) & -2Re(\lambda \overline{\omega})\\
	\frac{-1}{\sqrt{3}}Im(\lambda) & \frac{-1}{\sqrt{3}}Im(\lambda) & \frac{-1}{\sqrt{3}}Im(\lambda \omega) & \frac{-1}{\sqrt{3}}Im(\lambda \omega) & 2Re(\lambda) & 0 & -2Re(\lambda \omega)\\
	\frac{1}{\sqrt{3}}Im(\lambda \omega) & \frac{1}{\sqrt{3}}Im(\lambda \omega) & \frac{-1}{\sqrt{3}}Im(\lambda \overline{\omega}) & \frac{-1}{\sqrt{3}}Im(\lambda \overline{\omega}) & -2Re(\lambda \overline{\omega}) & -2Re(\lambda \omega) & 0\\
\end{array}\right]$$
with eigenvalue $\dfrac{1}{7^2}$ where $\xi \in P_{2,+}^{Spin}$ by the identification $\xi=\sum_{i,j=1}^7 \xi_{i,j} \hat{i} \otimes \hat{j}$. Thus every subfactor from this continuous family is infinite depth.
\end{exmp}

This eigenvector was found numerically in Matlab and verified using the Symbolic Math Toolbox. The Matlab code used can be found in the appendix.

\begin{exmp}
Let $p$ be a prime and $m \in \mathbb N$ such that $p^m \equiv 1 \text{ mod }4$. Then the Galois field of order $q=p^m$, $\mathbb F_q$, has a quadratic character
$$\chi(a)=\left\lbrace
\begin{array}{ccc} 0 & \text{if }a=0 \\
1 & \text{if } a=b^2 \text{ for some } b \in \mathbb F_q \backslash \{0\} \\
-1 & \text{if } a \neq b^2 \text{ for any } b \in \mathbb F_q \backslash \{0\} \\
\end{array} \right. .$$
Let $j_{n,m}$ be the $n \times m$ matrix of ones, $I_n$ the $n \times n$ identity matrix, and define the $q \times q$ matrix, $K_{a,b} = \chi(a-b)$, indexed by $\mathbb F_q$. Then the $2(q+1) \times 2(q+1)$ Paley type $II$ Hadamard matrix (\cite{Pal33}) is given by
$$H=\left[\begin{array}{cc} 0 & j\\ j^T & K\\ \end{array}\right] \otimes \left[\begin{array}{cc} 1 & 1\\ 1 & -1\\ \end{array}\right]+I_{q+1} \otimes \left[\begin{array}{cc} 1 & -1\\ -1 & -1\\ \end{array}\right].$$
Letting $u$ be the corresponding biunitary from $H$ and $Q=2(q+1)$, $Q\theta_u|_{P_{2,+}^{Spin}}$ has an eigenvector
$$\xi = \left[\begin{array}{cc} 0 & 0\\ 0 & K\\ \end{array}\right] \otimes \left[\begin{array}{cc} 1 & 0\\ 0 & -1\\ \end{array}\right]$$
with eigenvalue $\lambda=\dfrac{4^2q}{(q+1)^2}$. Since $q$ is a prime power congruent to $1 \text{ mod }4$, $\lambda$ is not an algebraic integer and so Paley type $II$ Hadamard matrices yield infinite depth subfactors.
\end{exmp}

\begin{proof}
Since the type $II$ Paley Hadamard matrices are more easily expressed using tensors, we will work in a tensor product of planar algebras as defined in \cite{Jon99}. Letting $P^{Spin,Q}$ denote the spin planar algebra with $Q$ spins, it can be shown that $P^{Spin,q+1} \otimes P^{Spin,2} \cong P^{Spin,2(q+1)}$ by a bijection, $\set{\hat{i} \otimes \hat{j}}{i=1,...,q+1, j=1,2} \leftrightarrow \set{\hat{k}}{k=1,...,2(q+1)}$. We will also identify matrices with the $2$-box spaces of $P^{Spin}$ via $(a_{i,j})_{i,j=1}^Q=\sum_{i,j}a_{i,j}\hat{i}\otimes\hat{j}$.

Define the $2\times 2$ and $q+1 \times q+1$ matrices
$$H_+=\left[\begin{array}{cc} 1 & 1\\ 1 & -1\\ \end{array}\right] \quad H_-=\left[\begin{array}{cc} 1 & -1\\ -1 & -1\\ \end{array}\right] \quad J=\left[\begin{array}{cc} 0 & j_{1,q}\\ j_{q,1} & 0\\ \end{array}\right] \quad L=\left[\begin{array}{cc} 0 & 0\\ j_{q,1} & 0\\ \end{array}\right]$$
$$D=\left[\begin{array}{cc} 1 & 0\\ 0 & -1\\ \end{array}\right] \quad E=\left[\begin{array}{cc} 0 & 1\\ -1 & 0\\ \end{array}\right] \quad T=\left[\begin{array}{cc} 0 & 0\\ 0 & j_{q,q}-I_q\\ \end{array}\right] \quad \text{and} \quad S=\left[\begin{array}{cc} 0 & 0\\ 0 & K\\ \end{array}\right].$$
Then $H=I_{q+1}\otimes H_- + J \otimes H_+ + S \otimes H_+$ and $\xi = S \otimes D$. Since $q \cong 1\text{ mod }4$, $K$ is symmetric, and so $H=H^*=H^T=\overline{H}$ and $\xi=\xi^*=\xi^T=\overline{\xi}$. This implies that orienting strings and marked intervals for $\xi$ are unnecessary. To evaluate $Q\theta_u(\xi)$, we first compute
$$\phi_u(\xi)=\begin{tikzpicture}[align.7]
	\clip[rounded corners] (0,-1.5) rectangle (3,1.5);
	\path[shaded] (0,-1.5) rectangle (3,1.5);
	\path[unshaded] (1.3,-1.5)--(1.3,1.5)--(1.7,1.5)--(1.7,-1.5);
	\path[unshaded] (1.5,0) circle (.9cm);
		\begin{scope}
		\clip (1.5,0) circle (.9cm);
		\path[shaded] (1.3,-1.5)--(1.3,1.5)--(1.7,1.5)--(1.7,-1.5);
		\end{scope}
	\draw (1.3,-1.5)--(1.3,1.5);
	\draw (1.7,-1.5)--(1.7,1.5);
	\draw (1.5,0) circle (.9cm);
	\node at (1.5,0)[Tbox, minimum width=.5cm, minimum height=.5cm]{\scriptsize $\xi$};
	\node at (1,0) {\scriptsize $\$$};
	\node at (.15,0) {\scriptsize $\$$};
	\draw[very thick, rounded corners] (0,-1.5) rectangle (3,1.5);
\end{tikzpicture}.$$
For each intersection of strings we must substitute in $H$. Since $H$ is a sum of three simple tensors, $\phi_u(\xi)$ decomposes into a sum of $3^4$ tangles with disks filled by the terms $S\otimes D$, $I_{q+1}\otimes H_-$, $J \otimes H_+$, or $S \otimes H_+$. The following identities force all but eight of these terms to be zero.\\
	$$(i) \space I_{q+1}=\dfrac{1}{\sqrt{q+1}}\begin{tikzpicture}[align1] \clip[rounded corners] (0,-.5) rectangle (1,.5); \draw[shaded] (.3,-.5) rectangle (.7,.5); \draw[rounded corners, very thick] (0,-.5) rectangle (1,.5); \end{tikzpicture}\in P^{Spin,q+1}_{2,+} \quad \quad
	(ii) \space \begin{tikzpicture}[align1]
		\clip[rounded corners] (0,-.6) rectangle (1,.6);
		\draw[shaded] (.35,-.6) rectangle (1,.6);
		\draw[unshaded, rounded corners] (.55,-.4) rectangle (.85,.4);
		\node at (.45,0)[Tbox, minimum width=.5cm, minimum height=.5cm]{\scriptsize $S$};
		\node at (.1,.45) {\scriptsize $\$$};
		\draw[very thick, rounded corners] (0,-.6) rectangle (1,.6);
	\end{tikzpicture}=
	\begin{tikzpicture}[align1]
		\clip[rounded corners] (0,-.6) rectangle (1,.6);
		\draw[shaded] (.35,-.6) rectangle (1,.6);
		\draw[unshaded, rounded corners] (.55,-.4) rectangle (.85,.4);
		\node at (.45,0)[Tbox, minimum width=.5cm, minimum height=.5cm]{\scriptsize $J$};
		\node at (.1,.45) {\scriptsize $\$$};
		\draw[very thick, rounded corners] (0,-.6) rectangle (1,.6);
	\end{tikzpicture}=
	\begin{tikzpicture}[align1]
		\clip[rounded corners] (0,-.6) rectangle (1,.6);
		\draw[shaded] (.35,-.6) rectangle (1,.6);
		\draw[unshaded, rounded corners] (.55,-.4) rectangle (.85,.4);
		\node at (.45,0)[Tbox, minimum width=.5cm, minimum height=.5cm]{\scriptsize $L$};
		\node at (.1,.45) {\scriptsize $\$$};
		\draw[very thick, rounded corners] (0,-.6) rectangle (1,.6);
	\end{tikzpicture}=
	\begin{tikzpicture}[align1]
		\clip[rounded corners] (0,-.6) rectangle (1,.6);
		\draw[shaded] (.35,-.6) rectangle (1,.6);
		\draw[unshaded, rounded corners] (.55,-.4) rectangle (.85,.4);
		\node at (.45,0)[Tbox, minimum width=.5cm, minimum height=.5cm]{\scriptsize $T$};
		\node at (.1,.45) {\scriptsize $\$$};
		\draw[very thick, rounded corners] (0,-.6) rectangle (1,.6);
	\end{tikzpicture}=0$$
	$$(iii) \space \begin{tikzpicture}[align1]
		\clip[rounded corners] (0,-.6) rectangle (2.5,.6);
		\draw[shaded] (0,-.6) rectangle (.525,.6);
		\draw[shaded] (.725,-.6) rectangle (1.15,.6);
		\draw[shaded] (1.35,-.6) rectangle (1.775,.6);
		\draw[shaded] (1.975,-.6) rectangle (2.5,.6);
		\node at (.625,0)[Tbox, minimum width=.5cm, minimum height=.5cm]{\scriptsize $S$};
		\node at (1.25,0)[Tbox, minimum width=.5cm, minimum height=.5cm]{\scriptsize $J$};
		\node at (1.875,0)[Tbox, minimum width=.5cm, minimum height=.5cm]{\scriptsize $S$};
		\node at (.1,.45) {\scriptsize $\$$};
		\draw[very thick, rounded corners] (0,-.6) rectangle (2.5,.6);
	\end{tikzpicture}=0\quad
	(iv) \space \begin{tikzpicture}[align1]
		\clip[rounded corners] (0,-1) rectangle (2,1);
		\path[shaded] (.5,-.5) rectangle (1.5,.5);
		\path[shaded] (0,-1) rectangle (.5,1);
		\path[shaded] (1.5,-1) rectangle (2,1);
		\path[unshaded] (1,0)--(.5,.5)--(.5,-.5)--(1,0);
		\path[unshaded] (1,0)--(1.5,.5)--(1.5,-.5)--(1,0);
		\draw (.5,-.5) rectangle (1.5,.5);
		\draw (1,0)--(.5,.5)--(.5,1)--(1.5,1)--(1.5,.5)--(1,0);
		\draw (1,0)--(.5,-.5)--(.5,-1)--(1.5,-1)--(1.5,-.5)--(1,0);
		\node at (1,0)[Tbox, minimum width=.5cm, minimum height=.5cm]{\scriptsize $D$};
		\node at (.5,.5)[Tbox, minimum width=.5cm, minimum height=.5cm]{\scriptsize $H_+$};
		\node at (.5,-.5)[Tbox, minimum width=.5cm, minimum height=.5cm]{\scriptsize $H_+$};
		\node at (1.5,.5)[Tbox, minimum width=.5cm, minimum height=.5cm]{\scriptsize $H_+$};
		\node at (1.5,-.5)[Tbox, minimum width=.5cm, minimum height=.5cm]{\scriptsize $H_+$};
		\node at (.1,.85) {\scriptsize $\$$};
		\draw[very thick, rounded corners] (0,-1) rectangle (2,1);
	\end{tikzpicture}=
	\begin{tikzpicture}[align1]
		\clip[rounded corners] (0,-1) rectangle (2,1);
		\path[shaded] (.5,-.5) rectangle (1.5,.5);
		\path[shaded] (0,-1) rectangle (.5,1);
		\path[shaded] (1.5,-1) rectangle (2,1);
		\path[unshaded] (1,0)--(.5,.5)--(.5,-.5)--(1,0);
		\path[unshaded] (1,0)--(1.5,.5)--(1.5,-.5)--(1,0);
		\draw (.5,-.5) rectangle (1.5,.5);
		\draw (1,0)--(.5,.5)--(.5,1)--(1.5,1)--(1.5,.5)--(1,0);
		\draw (1,0)--(.5,-.5)--(.5,-1)--(1.5,-1)--(1.5,-.5)--(1,0);
		\node at (1,0)[Tbox, minimum width=.5cm, minimum height=.5cm]{\scriptsize $D$};
		\node at (.5,.5)[Tbox, minimum width=.5cm, minimum height=.5cm]{\scriptsize $H_-$};
		\node at (.5,-.5)[Tbox, minimum width=.5cm, minimum height=.5cm]{\scriptsize $H_+$};
		\node at (1.5,.5)[Tbox, minimum width=.5cm, minimum height=.5cm]{\scriptsize $H_+$};
		\node at (1.5,-.5)[Tbox, minimum width=.5cm, minimum height=.5cm]{\scriptsize $H_-$};
		\node at (.1,.85) {\scriptsize $\$$};
		\draw[very thick, rounded corners] (0,-1) rectangle (2,1);
	\end{tikzpicture}=0=
	\begin{tikzpicture}[align1]
		\clip[rounded corners] (0,-1) rectangle (2,1);
		\path[shaded] (.5,-.5) rectangle (1.5,.5);
		\path[shaded] (0,-1) rectangle (.5,1);
		\path[shaded] (1.5,-1) rectangle (2,1);
		\path[unshaded] (1,0)--(.5,.5)--(.5,-.5)--(1,0);
		\path[unshaded] (1,0)--(1.5,.5)--(1.5,-.5)--(1,0);
		\draw (.5,-.5) rectangle (1.5,.5);
		\draw (1,0)--(.5,.5)--(.5,1)--(1.5,1)--(1.5,.5)--(1,0);
		\draw (1,0)--(.5,-.5)--(.5,-1)--(1.5,-1)--(1.5,-.5)--(1,0);
		\node at (1,0)[Tbox, minimum width=.5cm, minimum height=.5cm]{\scriptsize $E$};
		\node at (.5,.5)[Tbox, minimum width=.5cm, minimum height=.5cm]{\scriptsize $H_+$};
		\node at (.5,-.5)[Tbox, minimum width=.5cm, minimum height=.5cm]{\scriptsize $H_+$};
		\node at (1.5,.5)[Tbox, minimum width=.5cm, minimum height=.5cm]{\scriptsize $H_+$};
		\node at (1.5,-.5)[Tbox, minimum width=.5cm, minimum height=.5cm]{\scriptsize $H_+$};
		\node at (.1,.85) {\scriptsize $\$$};
		\draw[very thick, rounded corners] (0,-1) rectangle (2,1);
	\end{tikzpicture}=
	\begin{tikzpicture}[align1]
		\clip[rounded corners] (0,-1) rectangle (2,1);
		\path[shaded] (.5,-.5) rectangle (1.5,.5);
		\path[shaded] (0,-1) rectangle (.5,1);
		\path[shaded] (1.5,-1) rectangle (2,1);
		\path[unshaded] (1,0)--(.5,.5)--(.5,-.5)--(1,0);
		\path[unshaded] (1,0)--(1.5,.5)--(1.5,-.5)--(1,0);
		\draw (.5,-.5) rectangle (1.5,.5);
		\draw (1,0)--(.5,.5)--(.5,1)--(1.5,1)--(1.5,.5)--(1,0);
		\draw (1,0)--(.5,-.5)--(.5,-1)--(1.5,-1)--(1.5,-.5)--(1,0);
		\node at (1,0)[Tbox, minimum width=.5cm, minimum height=.5cm]{\scriptsize $E$};
		\node at (.5,.5)[Tbox, minimum width=.5cm, minimum height=.5cm]{\scriptsize $H_-$};
		\node at (.5,-.5)[Tbox, minimum width=.5cm, minimum height=.5cm]{\scriptsize $H_+$};
		\node at (1.5,.5)[Tbox, minimum width=.5cm, minimum height=.5cm]{\scriptsize $H_+$};
		\node at (1.5,-.5)[Tbox, minimum width=.5cm, minimum height=.5cm]{\scriptsize $H_-$};
		\node at (.1,.85) {\scriptsize $\$$};
		\draw[very thick, rounded corners] (0,-1) rectangle (2,1);
	\end{tikzpicture}$$
	$$(v) \space \begin{tikzpicture}[align1]
		\clip[rounded corners] (0,-1) rectangle (2,1);
		\path[shaded] (0,-1) rectangle (2,1);
		\path[unshaded, rounded corners] (.5,.5) rectangle (1.5,1);
		\path[unshaded, rounded corners] (.5,-.5) rectangle (1.5,-1);
		\path[unshaded, rounded corners] (1,0)--(.5,.1)--(.5,-.5)--(1,0);
		\path[unshaded, rounded corners] (1,0)--(1.5,.1)--(1.5,-.5)--(1,0);
		\draw[rounded corners] (.5,.5) rectangle (1.5,2);
		\draw[rounded corners] (.5,-.5) rectangle (1.5,-2);
		\draw[rounded corners] (1,0)--(.5,.1)--(.5,-.5)--(1,0);
		\draw[rounded corners] (1,0)--(1.5,.1)--(1.5,-.5)--(1,0);
		\node at (1,0)[Tbox, minimum width=.5cm, minimum height=.5cm]{\scriptsize $S$};
		\node at (.5,-.5)[Tbox, minimum width=.5cm, minimum height=.5cm]{\scriptsize $S$};
		\node at (1.5,-.5)[Tbox, minimum width=.5cm, minimum height=.5cm]{\scriptsize $S$};
		\node at (.1,.85) {\scriptsize $\$$};
		\draw[very thick, rounded corners] (0,-1) rectangle (2,1);
	\end{tikzpicture}=\sum_{i,r \in \mathbb F_q}\chi(i-r)^3 \hat{i} \otimes \hat{i} =0 \quad \quad
	(vi) \space \begin{tikzpicture}[align1]
		\clip[rounded corners] (0,-1) rectangle (2,1);
		\path[shaded] (.5,-.5) rectangle (1.5,.5);
		\path[shaded] (0,-1) rectangle (.5,1);
		\path[shaded] (1.5,-1) rectangle (2,1);
		\path[unshaded] (1,0)--(.5,.5)--(.5,-.5)--(1,0);
		\path[unshaded] (1,0)--(1.5,.5)--(1.5,-.5)--(1,0);
		\draw (.5,-.5) rectangle (1.5,.5);
		\draw (1,0)--(.5,.5)--(.5,1)--(1.5,1)--(1.5,.5)--(1,0);
		\draw (1,0)--(.5,-.5)--(.5,-1)--(1.5,-1)--(1.5,-.5)--(1,0);
		\node at (1,0)[Tbox, minimum width=.5cm, minimum height=.5cm]{\scriptsize $D$};
		\node at (.5,.5)[Tbox, minimum width=.5cm, minimum height=.5cm]{\scriptsize $H_-$};
		\node at (.5,-.5)[Tbox, minimum width=.5cm, minimum height=.5cm]{\scriptsize $H_+$};
		\node at (1.5,.5)[Tbox, minimum width=.5cm, minimum height=.5cm]{\scriptsize $H_+$};
		\node at (1.5,-.5)[Tbox, minimum width=.5cm, minimum height=.5cm]{\scriptsize $H_+$};
		\node at (.1,.85) {\scriptsize $\$$};
		\draw[very thick, rounded corners] (0,-1) rectangle (2,1);
	\end{tikzpicture}=\left[\begin{array}{cc} 1 & 1\\ -1 & -1\\ \end{array}\right]$$
Due to $(iv)$, terms of $\phi_u(\xi)$ without $I_{q+1} \otimes H_-$ are zero. By $(i)$, terms with two or more $I_{q+1} \otimes H_-$'s are zero due to $(ii)$, $(iii)$, $(iv)$, or $(v)$ depending on the placement of $I_{q+1} \otimes H_-$ terms. Thus all nonzero terms contain one instance of $I_{q+1} \otimes H_-$. Due to $(iii)$, terms with exactly one $I_{q+1} \otimes H_-$ and one $J \otimes H_+$ disks are zero. $(iii)$ further restricts how $J \otimes H_+$ and $S \otimes H_+$ can be arranged to yield nonzero terms. Therefore the only nonzero terms are
$$\dfrac{1}{\sqrt{q+1}}
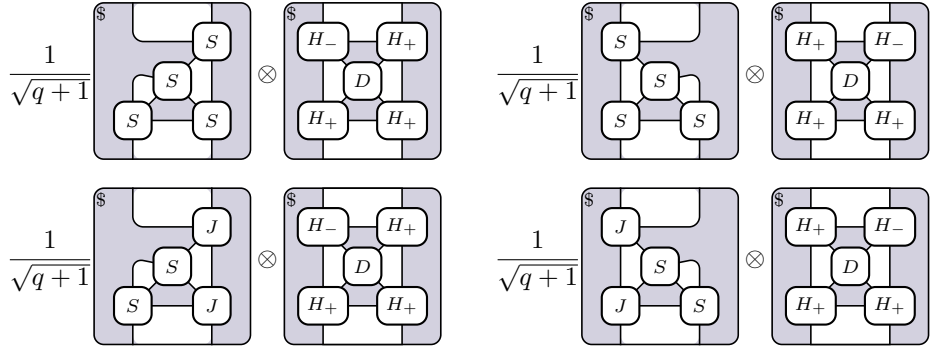
\begin{tikzpicture}[align1]
	\clip[rounded corners] (0,-1) rectangle (2,1);
	\path[shaded] (0,-1) rectangle (2,1);
	\path[unshaded, rounded corners] (.5,.5) rectangle (1.5,1);
	\path[unshaded, rounded corners] (.5,-.5) rectangle (1.5,-1);
	\path[unshaded, rounded corners] (1,0)--(.5,.1)--(.5,-.5)--(1,0);
	\path[unshaded, rounded corners] (1,0)--(1.5,.5)--(1.5,-.5)--(1,0);
	\draw[rounded corners] (.5,.5) rectangle (1.5,2);
	\draw[rounded corners] (.5,-.5) rectangle (1.5,-2);
	\draw[rounded corners] (1,0)--(.5,.1)--(.5,-.5)--(1,0);
	\draw[rounded corners] (1,0)--(1.5,.5)--(1.5,-.5)--(1,0);
	\node at (1,0)[Tbox, minimum width=.5cm, minimum height=.5cm]{\scriptsize $S$};
	\node at (.5,-.5)[Tbox, minimum width=.5cm, minimum height=.5cm]{\scriptsize $S$};
	\node at (1.5,.5)[Tbox, minimum width=.5cm, minimum height=.5cm]{\scriptsize $S$};
	\node at (1.5,-.5)[Tbox, minimum width=.5cm, minimum height=.5cm]{\scriptsize $S$};
	\node at (.1,.85) {\scriptsize $\$$};
	\draw[very thick, rounded corners] (0,-1) rectangle (2,1);
\end{tikzpicture} \otimes
\begin{tikzpicture}[align1]
	\clip[rounded corners] (0,-1) rectangle (2,1);
	\path[shaded] (.5,-.5) rectangle (1.5,.5);
	\path[shaded] (0,-1) rectangle (.5,1);
	\path[shaded] (1.5,-1) rectangle (2,1);
	\path[unshaded] (1,0)--(.5,.5)--(.5,-.5)--(1,0);
	\path[unshaded] (1,0)--(1.5,.5)--(1.5,-.5)--(1,0);
	\draw (.5,-.5) rectangle (1.5,.5);
	\draw (1,0)--(.5,.5)--(.5,1)--(1.5,1)--(1.5,.5)--(1,0);
	\draw (1,0)--(.5,-.5)--(.5,-1)--(1.5,-1)--(1.5,-.5)--(1,0);
	\node at (1,0)[Tbox, minimum width=.5cm, minimum height=.5cm]{\scriptsize $D$};
	\node at (.5,.5)[Tbox, minimum width=.5cm, minimum height=.5cm]{\scriptsize $H_-$};
	\node at (.5,-.5)[Tbox, minimum width=.5cm, minimum height=.5cm]{\scriptsize $H_+$};
	\node at (1.5,.5)[Tbox, minimum width=.5cm, minimum height=.5cm]{\scriptsize $H_+$};
	\node at (1.5,-.5)[Tbox, minimum width=.5cm, minimum height=.5cm]{\scriptsize $H_+$};
	\node at (.1,.85) {\scriptsize $\$$};
	\draw[very thick, rounded corners] (0,-1) rectangle (2,1);
\end{tikzpicture}
\quad \quad
\dfrac{1}{\sqrt{q+1}}\begin{tikzpicture}[align1]
	\clip[rounded corners] (0,-1) rectangle (2,1);
	\path[shaded] (0,-1) rectangle (2,1);
	\path[unshaded, rounded corners] (.5,.5) rectangle (1.5,1);
	\path[unshaded, rounded corners] (.5,-.5) rectangle (1.5,-1);
	\path[unshaded, rounded corners] (1,0)--(.5,.5)--(.5,-.5)--(1,0);
	\path[unshaded, rounded corners] (1,0)--(1.5,.1)--(1.5,-.5)--(1,0);
	\draw[rounded corners] (.5,.5) rectangle (1.5,2);
	\draw[rounded corners] (.5,-.5) rectangle (1.5,-2);
	\draw[rounded corners] (1,0)--(.5,.5)--(.5,-.5)--(1,0);
	\draw[rounded corners] (1,0)--(1.5,.1)--(1.5,-.5)--(1,0);
	\node at (1,0)[Tbox, minimum width=.5cm, minimum height=.5cm]{\scriptsize $S$};
	\node at (.5,-.5)[Tbox, minimum width=.5cm, minimum height=.5cm]{\scriptsize $S$};
	\node at (.5,.5)[Tbox, minimum width=.5cm, minimum height=.5cm]{\scriptsize $S$};
	\node at (1.5,-.5)[Tbox, minimum width=.5cm, minimum height=.5cm]{\scriptsize $S$};
	\node at (.1,.85) {\scriptsize $\$$};
	\draw[very thick, rounded corners] (0,-1) rectangle (2,1);
\end{tikzpicture} \otimes 
\begin{tikzpicture}[align1]
	\clip[rounded corners] (0,-1) rectangle (2,1);
	\path[shaded] (.5,-.5) rectangle (1.5,.5);
	\path[shaded] (0,-1) rectangle (.5,1);
	\path[shaded] (1.5,-1) rectangle (2,1);
	\path[unshaded] (1,0)--(.5,.5)--(.5,-.5)--(1,0);
	\path[unshaded] (1,0)--(1.5,.5)--(1.5,-.5)--(1,0);
	\draw (.5,-.5) rectangle (1.5,.5);
	\draw (1,0)--(.5,.5)--(.5,1)--(1.5,1)--(1.5,.5)--(1,0);
	\draw (1,0)--(.5,-.5)--(.5,-1)--(1.5,-1)--(1.5,-.5)--(1,0);
	\node at (1,0)[Tbox, minimum width=.5cm, minimum height=.5cm]{\scriptsize $D$};
	\node at (.5,.5)[Tbox, minimum width=.5cm, minimum height=.5cm]{\scriptsize $H_+$};
	\node at (.5,-.5)[Tbox, minimum width=.5cm, minimum height=.5cm]{\scriptsize $H_+$};
	\node at (1.5,.5)[Tbox, minimum width=.5cm, minimum height=.5cm]{\scriptsize $H_-$};
	\node at (1.5,-.5)[Tbox, minimum width=.5cm, minimum height=.5cm]{\scriptsize $H_+$};
	\node at (.1,.85) {\scriptsize $\$$};
	\draw[very thick, rounded corners] (0,-1) rectangle (2,1);
\end{tikzpicture}$$
$$\dfrac{1}{\sqrt{q+1}}\begin{tikzpicture}[align1]
	\clip[rounded corners] (0,-1) rectangle (2,1);
	\path[shaded] (0,-1) rectangle (2,1);
	\path[unshaded, rounded corners] (.5,.5) rectangle (1.5,1);
	\path[unshaded, rounded corners] (.5,-.5) rectangle (1.5,-1);
	\path[unshaded, rounded corners] (1,0)--(.5,.1)--(.5,-.5)--(1,0);
	\path[unshaded, rounded corners] (1,0)--(1.5,.5)--(1.5,-.5)--(1,0);
	\draw[rounded corners] (.5,.5) rectangle (1.5,2);
	\draw[rounded corners] (.5,-.5) rectangle (1.5,-2);
	\draw[rounded corners] (1,0)--(.5,.1)--(.5,-.5)--(1,0);
	\draw[rounded corners] (1,0)--(1.5,.5)--(1.5,-.5)--(1,0);
	\node at (1,0)[Tbox, minimum width=.5cm, minimum height=.5cm]{\scriptsize $S$};
	\node at (.5,-.5)[Tbox, minimum width=.5cm, minimum height=.5cm]{\scriptsize $S$};
	\node at (1.5,.5)[Tbox, minimum width=.5cm, minimum height=.5cm]{\scriptsize $J$};
	\node at (1.5,-.5)[Tbox, minimum width=.5cm, minimum height=.5cm]{\scriptsize $J$};
	\node at (.1,.85) {\scriptsize $\$$};
	\draw[very thick, rounded corners] (0,-1) rectangle (2,1);
\end{tikzpicture} \otimes 
\begin{tikzpicture}[align1]
	\clip[rounded corners] (0,-1) rectangle (2,1);
	\path[shaded] (.5,-.5) rectangle (1.5,.5);
	\path[shaded] (0,-1) rectangle (.5,1);
	\path[shaded] (1.5,-1) rectangle (2,1);
	\path[unshaded] (1,0)--(.5,.5)--(.5,-.5)--(1,0);
	\path[unshaded] (1,0)--(1.5,.5)--(1.5,-.5)--(1,0);
	\draw (.5,-.5) rectangle (1.5,.5);
	\draw (1,0)--(.5,.5)--(.5,1)--(1.5,1)--(1.5,.5)--(1,0);
	\draw (1,0)--(.5,-.5)--(.5,-1)--(1.5,-1)--(1.5,-.5)--(1,0);
	\node at (1,0)[Tbox, minimum width=.5cm, minimum height=.5cm]{\scriptsize $D$};
	\node at (.5,.5)[Tbox, minimum width=.5cm, minimum height=.5cm]{\scriptsize $H_-$};
	\node at (.5,-.5)[Tbox, minimum width=.5cm, minimum height=.5cm]{\scriptsize $H_+$};
	\node at (1.5,.5)[Tbox, minimum width=.5cm, minimum height=.5cm]{\scriptsize $H_+$};
	\node at (1.5,-.5)[Tbox, minimum width=.5cm, minimum height=.5cm]{\scriptsize $H_+$};
	\node at (.1,.85) {\scriptsize $\$$};
	\draw[very thick, rounded corners] (0,-1) rectangle (2,1);
\end{tikzpicture}
\quad \quad
\dfrac{1}{\sqrt{q+1}}\begin{tikzpicture}[align1]
	\clip[rounded corners] (0,-1) rectangle (2,1);
	\path[shaded] (0,-1) rectangle (2,1);
	\path[unshaded, rounded corners] (.5,.5) rectangle (1.5,1);
	\path[unshaded, rounded corners] (.5,-.5) rectangle (1.5,-1);
	\path[unshaded, rounded corners] (1,0)--(.5,.5)--(.5,-.5)--(1,0);
	\path[unshaded, rounded corners] (1,0)--(1.5,.1)--(1.5,-.5)--(1,0);
	\draw[rounded corners] (.5,.5) rectangle (1.5,2);
	\draw[rounded corners] (.5,-.5) rectangle (1.5,-2);
	\draw[rounded corners] (1,0)--(.5,.5)--(.5,-.5)--(1,0);
	\draw[rounded corners] (1,0)--(1.5,.1)--(1.5,-.5)--(1,0);
	\node at (1,0)[Tbox, minimum width=.5cm, minimum height=.5cm]{\scriptsize $S$};
	\node at (.5,-.5)[Tbox, minimum width=.5cm, minimum height=.5cm]{\scriptsize $J$};
	\node at (.5,.5)[Tbox, minimum width=.5cm, minimum height=.5cm]{\scriptsize $J$};
	\node at (1.5,-.5)[Tbox, minimum width=.5cm, minimum height=.5cm]{\scriptsize $S$};
	\node at (.1,.85) {\scriptsize $\$$};
	\draw[very thick, rounded corners] (0,-1) rectangle (2,1);
\end{tikzpicture} \otimes 
\begin{tikzpicture}[align1]
	\clip[rounded corners] (0,-1) rectangle (2,1);
	\path[shaded] (.5,-.5) rectangle (1.5,.5);
	\path[shaded] (0,-1) rectangle (.5,1);
	\path[shaded] (1.5,-1) rectangle (2,1);
	\path[unshaded] (1,0)--(.5,.5)--(.5,-.5)--(1,0);
	\path[unshaded] (1,0)--(1.5,.5)--(1.5,-.5)--(1,0);
	\draw (.5,-.5) rectangle (1.5,.5);
	\draw (1,0)--(.5,.5)--(.5,1)--(1.5,1)--(1.5,.5)--(1,0);
	\draw (1,0)--(.5,-.5)--(.5,-1)--(1.5,-1)--(1.5,-.5)--(1,0);
	\node at (1,0)[Tbox, minimum width=.5cm, minimum height=.5cm]{\scriptsize $D$};
	\node at (.5,.5)[Tbox, minimum width=.5cm, minimum height=.5cm]{\scriptsize $H_+$};
	\node at (.5,-.5)[Tbox, minimum width=.5cm, minimum height=.5cm]{\scriptsize $H_+$};
	\node at (1.5,.5)[Tbox, minimum width=.5cm, minimum height=.5cm]{\scriptsize $H_-$};
	\node at (1.5,-.5)[Tbox, minimum width=.5cm, minimum height=.5cm]{\scriptsize $H_+$};
	\node at (.1,.85) {\scriptsize $\$$};
	\draw[very thick, rounded corners] (0,-1) rectangle (2,1);
\end{tikzpicture}$$
and their conjugates. These terms can be evaluated directly and yield $$\phi_u(\xi)=-\dfrac{4}{q+1}T\otimes D + \dfrac{2(q-1)}{q+1}\left( J \otimes D +L \otimes E-L^T \otimes E\right).$$
Letting
$$\phi_u^*(\xi)=\begin{tikzpicture}[align.7]
	\clip[rounded corners] (0,-1.5) rectangle (3,1.5);
	\path[unshaded] (0,-1.5) rectangle (3,1.5);
	\path[shaded] (1.3,-1.5)--(1.3,1.5)--(1.7,1.5)--(1.7,-1.5);
	\path[shaded] (1.5,0) circle (.9cm);
	\begin{scope}
		\clip (1.5,0) circle (.9cm);
		\path[unshaded] (1.3,-1.5)--(1.3,1.5)--(1.7,1.5)--(1.7,-1.5);
	\end{scope}
	\draw (1.3,-1.5)--(1.3,1.5);
	\draw (1.7,-1.5)--(1.7,1.5);
	\draw (1.5,0) circle (.9cm);
	\node at (1.5,0)[Tbox, minimum width=.5cm, minimum height=.5cm]{\scriptsize $\xi$};
	\node at (1,0) {\scriptsize $\$$};
	\node at (.15,0) {\scriptsize $\$$};
	\draw[very thick, rounded corners] (0,-1.5) rectangle (3,1.5);
\end{tikzpicture}$$
we must evaluate $\phi_u^*(T\otimes D)$, $\phi_u^*(J \otimes D)$, $\phi_u^*(L \otimes E)$, and $\phi_u^*(L^T \otimes E)$. Fortunately $T$, $J$, $L$, and $E$ satisfy similar identities to $S$ and $D$ forcing most of the terms in $\phi^*_u$ to be zero. Only one term requires a nontrivial fact of $K$ which we highlight here. It can be shown that $K^2=qI_q - j_{q,q}$ using basic properties of $\mathbb F_q$ and $\chi$. We use this identity to evaluate the following term appearing in $\phi^*_u(T \otimes D)$.
$$\dfrac{1}{\sqrt{q+1}}\begin{tikzpicture}[align1]
	\clip[rounded corners] (0,-1) rectangle (2,1);
	\path[shaded] (.5,-1) rectangle (1.5,1);
	\path[unshaded] (1,0)--(.5,-.5)--(1.5,-.5)--cycle;
	\path[unshaded, rounded corners] (1,0)--(.7,.5)--(1.5,.5)--cycle;
	\draw (.5,-1) rectangle (1.5,1);
	\draw (1,0)--(.5,-.5)--(1.5,-.5)--cycle;
	\draw[rounded corners] (1,0)--(.7,.5)--(1.5,.5)--cycle;
	\node at (1,0)[Tbox, minimum width=.5cm, minimum height=.5cm]{\scriptsize $T$};
	\node at (.5,-.5)[Tbox, minimum width=.5cm, minimum height=.5cm]{\scriptsize $S$};
	\node at (1.5,.5)[Tbox, minimum width=.5cm, minimum height=.5cm]{\scriptsize $S$};
	\node at (1.5,-.5)[Tbox, minimum width=.5cm, minimum height=.5cm]{\scriptsize $S$};
	\node at (.1,.85) {\scriptsize $\$$};
	\draw[very thick, rounded corners] (0,-1) rectangle (2,1);
\end{tikzpicture} \otimes 
\begin{tikzpicture}[align1]
	\clip[rounded corners] (0,-1) rectangle (2,1);
	\path[shaded] (.5,-1) rectangle (1.5,1);
	\path[unshaded] (1,0)--(.5,-.5)--(1.5,-.5)--cycle;
	\path[unshaded] (1,0)--(.5,.5)--(1.5,.5)--cycle;
	\draw (.5,-1) rectangle (1.5,1);
	\draw (1,0)--(.5,-.5)--(1.5,-.5)--cycle;
	\draw (1,0)--(.5,.5)--(1.5,.5)--cycle;
	\node at (1,0)[Tbox, minimum width=.5cm, minimum height=.5cm]{\scriptsize $D$};
	\node at (.5,.5)[Tbox, minimum width=.5cm, minimum height=.5cm]{\scriptsize $H_-$};
	\node at (.5,-.5)[Tbox, minimum width=.5cm, minimum height=.5cm]{\scriptsize $H_+$};
	\node at (1.5,.5)[Tbox, minimum width=.5cm, minimum height=.5cm]{\scriptsize $H_+$};
	\node at (1.5,-.5)[Tbox, minimum width=.5cm, minimum height=.5cm]{\scriptsize $H_+$};
	\node at (.1,.85) {\scriptsize $\$$};
	\draw[very thick, rounded corners] (0,-1) rectangle (2,1);
\end{tikzpicture}=
\dfrac{1}{\sqrt{q+1}}\begin{tikzpicture}[align1]
	\clip[rounded corners] (0,-1) rectangle (2,1);
	\path[shaded] (.5,-1) rectangle (1.5,1);
	\path[unshaded, rounded corners] (.9,0)--(.7,.5)--(1.5,.5)--(1.1,0)--(1.5,-.5)--(.5,-.5)--cycle;
	\draw (.5,-1) rectangle (1.5,1);
	\draw[rounded corners] (.9,0)--(.7,.5)--(1.5,.5)--(1.1,0)--(1.5,-.5)--(.5,-.5)--cycle;
	\node at (.5,-.5)[Tbox, minimum width=.5cm, minimum height=.5cm]{\scriptsize $S$};
	\node at (1.5,.5)[Tbox, minimum width=.5cm, minimum height=.5cm]{\scriptsize $S$};
	\node at (1.5,-.5)[Tbox, minimum width=.5cm, minimum height=.5cm]{\scriptsize $S$};
	\node at (.1,.85) {\scriptsize $\$$};
	\draw[very thick, rounded corners] (0,-1) rectangle (2,1);
\end{tikzpicture} \otimes \left[\begin{array}{cc} 1 & -1\\ 1 & -1\\ \end{array}\right]$$
$$=\dfrac{q}{(q+1)^{3/2}}\begin{tikzpicture}[align1]
	\clip[rounded corners] (0,-.5) rectangle (1.3,.5);
	\draw[shaded] (.3,-1) rectangle (1.1,1);
	\draw[unshaded] (.7,0) circle (.2cm);
	\node at (.4,0)[Tbox, minimum width=.5cm, minimum height=.5cm]{\scriptsize $S$};
	\node at (.1,.35) {\scriptsize $\$$};
	\draw[very thick, rounded corners] (0,-.5) rectangle (1.3,.5);
\end{tikzpicture} \otimes \left[\begin{array}{cc} 1 & -1\\ 1 & -1\\ \end{array}\right] - \dfrac{1}{q+1} S \otimes \left[\begin{array}{cc} 1 & -1\\ 1 & -1\\ \end{array}\right] =- \dfrac{1}{q+1} S \otimes \left[\begin{array}{cc} 1 & -1\\ 1 & -1\\ \end{array}\right]$$
The remaining terms can be evaluated using similar techniques with the identity, $Kj_{q,1}=0$, and we list the results below.
$$\phi_u^*(T \otimes D)=-\dfrac{4}{q+1} S \otimes D \quad \phi_u^*(J \otimes D)=\dfrac{4}{q+1} S \otimes D \quad \phi_u^*(L \otimes E)=\dfrac{2}{q+1} S \otimes D=-\phi_u^*(L^T \otimes E)$$
Combining these computations yield $Q\theta_u(\xi)=\dfrac{4^2q}{(q+1)^2}\xi$.
\end{proof}

In \cite{Jon99} Jones shows that Paley type $I$ Hadamard matrix subfactors either have Temperley-Lieb $2$-box spaces or are depth two. In particular, the Paley type $I$ Hadamard matrices of dimension $12$ and $24$ have Temperley-Lieb $2$-box spaces. Similar examples are explored in \cite{Lia95}. Since the Paley type $I$ and type $II$ $12 \times 12$ Hadamard matrices are equivalent, the unique $12 \times 12$ Hadamard matrix yields an infinite depth subfactor with a Temperley-Lieb $2$-box space. From numerical computations we expect this to be true for Paley type $II$ Hadamard matrices as well. Since Paley type $II$ Hadamard matrix subfactors are infinite depth we suspect they have Temperley-Lieb standard invariants, hence $A_\infty$ principal graphs.

\begin{conj}
	Paley type $II$ Hadamard matrix subfactors have Temperley-Lieb standard invariants.
\end{conj}

\begin{center}
\large Acknowledgments
\end{center}

The author is very grateful to his advisor Professor Dietmar Bisch for the many conversations, excellent suggestions, and feedback on the preprint. The author also thanks Professor Larry Rolen for his encouragement and feedback on the Paley type $II$ Hadamard computations.

\pagebreak

\textsc{Department of Mathematics, Vanderbilt University, 1326 Stevenson Center, Nashville, TN 37240}

\textit{Email address:} \texttt{michael.r.montgomery@vanderbilt.edu}

\pagebreak
\begin{appendices}
	\section{}
	\lstset{language=Matlab,
		breaklines=true,
		morekeywords={matlab2tikz},
		keywordstyle=\color{blue},
		morekeywords=[2]{1}, keywordstyle=[2]{\color{black}},
		identifierstyle=\color{black},
		stringstyle=\color{mylilas},
		commentstyle=\color{mygreen},
		showstringspaces=false,
		numbers=left,
		numberstyle={\tiny \color{black}},
		numbersep=9pt,
		emph=[1]{for,end,break},emphstyle=[1]\color{red},   
	}
	\begin{lstlisting}
t=sym('t'); %This symbol stands for lambda.
w=sym(exp(1i*pi/3)); %w is the 6th primitive root of unity.
    u = [t*w t*w^4 w^5 w^3 w^3 w^1 1;
        t*w^4 t*w w^3 w^5 w^3 w^1 1;
        w^5 w^3 conj(t)*w conj(t)*w^4 w^1 w^3 1;
        w^3 w^5 conj(t)*w^4 conj(t)*w w^1 w^3 1;
        w^3 w^3 w w w^4 w^5 1;
        w w w^3 w^3 w^5 w^4 1;
        1 1 1 1 1 1 1]./sqrt(7);
Eigenvalue=sym('1/49');
EigenvectorArray=[0 0 1 -1 (1/sqrt(3))*imag(t*conj(w)) (-1/sqrt(3))*imag(t) (1/sqrt(3))*imag(t*w);
    0 0 -1 1 (1/sqrt(3))*imag(t*conj(w)) (-1/sqrt(3))*imag(t) (1/sqrt(3))*imag(t*w);
    1 -1 0 0 (1/sqrt(3))*imag(t) (-1/sqrt(3))*imag(t*w) (-1/sqrt(3))*imag(t*conj(w));
    -1 1 0 0 (1/sqrt(3))*imag(t) (-1/sqrt(3))*imag(t*w) (-1/sqrt(3))*imag(t*conj(w));
    (1/sqrt(3))*imag(t*conj(w)) (1/sqrt(3))*imag(t*conj(w)) (1/sqrt(3))*imag(t) (1/sqrt(3))*imag(t) 0 2*real(t) -2*real(t*conj(w));
    (-1/sqrt(3))*imag(t) (-1/sqrt(3))*imag(t) (-1/sqrt(3))*imag(t*w) (-1/sqrt(3))*imag(t*w) 2*real(t) 0 -2*real(t*w);
    (1/sqrt(3))*imag(t*w) (1/sqrt(3))*imag(t*w) (-1/sqrt(3))*imag(t*conj(w)) (-1/sqrt(3))*imag(t*conj(w)) -2*real(t*conj(w)) -2*real(t*w) 0];
Eigenvector=sym('e',[49 1]);
for a1=1:7    %Since we are representing everything on P_{2,+}^{Spin} we use 2-digit base 7 numbers for rows and columns.
    for a2=1:7
        Eigenvector(a1+7*(a2-1),1)=EigenvectorArray(a1,a2);
    end
end  %The definition of the profile matrix can be found in [Jon99].
AngleOp=sym('A',[7^2 7^2]);         %The angle operator as represented on P_{2,+}^{Spin}.
    for a1=1:7
        for a2=1:7
            for b1=1:7
                for b2=1:7
            r=sym(0);      %r is a running total for each entry of the profile matrix (Defined in [Jon99]).
                    for m=1:7
            r=r+u(m,b1)*conj(u(m,b2))*conj(u(m,a1))*u(m,a2);
                    end
            AngleOp(a1+7*(a2-1),b1+7*(b2-1))=7*r*conj(r);
                end
            end
        end
    end
Test=7*AngleOp*Eigenvector-Eigenvalue.*Eigenvector;  %If Test is zero then 1/49 is an eigenvalue of 7*AngleOp. 
Substitute1=subs(Test,real(t),(t+conj(t))/2);
Substitute2=subs(Substitute1,imag(t),(t-conj(t))/(2i));
TestExpanded=expand(Substitute2);
Substitute3=subs(TestExpanded,t*real(t),(1/2)*(t^2+1));
Substitute4=subs(Substitute3,t*conj(t),1); %Here we make several substitutions utilizing |t|=1 and expand terms.
if any(Substitute4~='0')
EigenvectorTest=false;
else
EigenvectorTest=true;
end
clear a1 a2 b1 b2 m r
\end{lstlisting}
\end{appendices}

\end{document}